\documentclass[11pt,english]{article}
\usepackage{mathptmx}

\usepackage[T1]{fontenc}
\usepackage[latin9]{inputenc}
\usepackage{geometry}
\usepackage{float}
\usepackage{amsthm}
\usepackage{amsmath}
\usepackage{amssymb}
\usepackage{graphicx}
\usepackage{esint}

\makeatletter
\theoremstyle{plain}
\newtheorem{thm}{\protect\theoremname}
  \theoremstyle{plain}
  \newtheorem{prop}[thm]{\protect\propositionname}
  \theoremstyle{plain}
  \newtheorem{cor}[thm]{\protect\corollaryname}

\makeatother

\usepackage{babel}
  \providecommand{\corollaryname}{Corollary}
  \providecommand{\propositionname}{Proposition}
\providecommand{\theoremname}{Theorem}

\begin{document}

\title{Deterministic Bayesian information fusion\\
 and the analysis of its performance}

\author{Gaurav Thakur%
\thanks{\noindent Digital Signal Corporation, 14000 Thunderbolt Place, Chantilly,
VA 20151, email: gthakur@alumni.princeton.edu. This work was done
while the author was at The MITRE Corporation, McLean, VA 22102. The
author would like to thank the Joint IED Defeat Organization (JIEDDO)
for supporting and funding this work, and Dr. Dave Colella and Dr.
Garry Jacyna of MITRE for providing valuable feedback and suggestions.
Approved for Public Release; Distribution Unlimited. 13-3248.%
}}

\date{October 16, 2014}
\maketitle
\begin{abstract}
This paper develops a mathematical and computational framework for
analyzing the expected performance of Bayesian data fusion, or joint
statistical inference, within a sensor network. We use variational
techniques to obtain the posterior expectation as the optimal fusion
rule under a deterministic constraint and a quadratic cost, and study
the smoothness and other properties of its classification performance.
For a certain class of fusion problems, we prove that this fusion
rule is also optimal in a much wider sense and satisfies strong asymptotic
convergence results. We show how these results apply to a variety
of examples with Gaussian, exponential and other statistics, and discuss
computational methods for determining the fusion system's performance
in more general, large-scale problems. These results are motivated
by studying the performance of fusing multi-modal radar and acoustic
sensors for detecting explosive substances, but have broad applicability
to other Bayesian decision problems.
\end{abstract}
Keywords: sensor data fusion, Bayes optimal decisions, computational
statistics, machine learning, calculus of variations, probabilistic
graphical models\\

AMS subject classification: 62C10, 49K30, 46N30

\section{Introduction\label{SecIntro}}

Sensor networks are ubiquitous across many different domains, including
wireless communications, temperature and process control, area surveillance,
object tracking and numerous other fields \cite{ASSC02,CK03}. Large
performance gains can be achieved in such networks by performing \textit{data
fusion} between the sensors, or combining information from the individual
sensors to reach system-level decisions \cite{Du07,Mi07,Va96,XT09}.
The sensors are typically connected by wireless links to either a
separate information collector (centralized fusion) or to each other
(distributed fusion). Elementary fusion rules based on Boolean logic
are used in many contexts due to their simplicity and ease of implementation.
On the other hand, in most situations we have some knowledge of the
statistical properties of the sensors' outputs, and designing fusion
rules that take this into account can provide much better performance
\cite{NLF07,Va96}. The fusion rule can be built to satisfy any of
various statistical optimality criteria, such as achieving the maximum
likelihood or the minimum Bayes risk, under any other constraints
of the problem \cite{NLF07}. Sensor information fusion can also be
understood as a special case of the more general problem of statistical
data reduction, where the goal is to reduce information from a high-dimensional
space into a low-dimensional one in some optimal manner.\\

In many sensor fusion applications, it is important for the fusion
rule to be a deterministic function of the sensor outputs. Fusion
techniques that incorporate randomness are widely used in the sensing
literature and can be more easily optimized to achieve given performance
targets, such as false or correct classification probabilities \cite{Po94}.
However, in certain applications such as the detection of explosive
compounds, it is common for the number of positive targets to be several
orders of magnitude smaller than the number of negative ones, which
means that randomized fusion rules have a large variance and rarely
achieve their expected theoretical performance on realistic sample
sizes. For simple, binary decision-level fusion problems, deterministic
rules are easy to find \cite{Va96}, but in general, requiring the
fusion rule to be deterministic effectively introduces a nonconvex
constraint that is difficult to incorporate into a numerical optimization
framework. This type of constraint also complicates the calculation
of a fusion rule's expected classification performance, which is a
key component of modeling and simulation efforts to design sensor
layouts and to perform trade studies between different sensor configurations
\cite{Th13}.\\

The goals of this paper are to study the fusion rule that is Bayes
optimal among all deterministic fusion rules, to investigate its mathematical
properties under different cost criteria and other problem constraints,
and to develop a computational framework for finding its expected
performance. These results are motivated by the standoff detection
of threat substances using multi-modal sensors in a centralized fusion
network. However, we formulate the problem in an abstract setting
that makes minimal assumptions on the details of the sensors or the
type of information they produce, in contrast to the relatively well-defined
situations in much of the sensing literature (e.g., \cite{CJKV04,Du07,SS06}).
We first use variational techniques to derive the standard Bayesian
posterior mean as the optimal deterministic fusion rule under a quadratic
cost. We show that the resulting system's classification performance
is a smooth function under some regularity constraints on the problem,
and describe extensions to more general settings where the posterior
mean no longer applies. For a certain class of fusion scenarios with
sub-Gaussian priors, we extend these results beyond the classical
setting of a quadratic cost, showing that the same fusion rule remains
optimal under higher order cost functions and that its classification
performance exhibits stronger, pointwise-like asymptotic behavior.
This mathematical theory is developed in Section \ref{SecMain}, with
the proofs of the theorems deferred to Appendix A. In Section \ref{SecExamples},
we apply these results to several illustrative examples with Gaussian,
exponential and other statistics for the sensors and examine the different
types of behavior that they exhibit. In Section \ref{SecNumerical},
we finally discuss efficient computational methods for finding the
performance of the fusion rule in practice, based on Monte Carlo integration
techniques commonly used in machine learning and other fields.

\section{Mathematical framework for data fusion\label{SecMain}}

Suppose we have $M$ sensors (random variables) $\{A_{m}\}_{1\leq m\le M}$,
each observing an object (hypothesis) $H$ and producing outputs that
are sent into a fusion center $C$ (see Figure \ref{FigFusion}).
The sensor outputs could represent simple true/false decisions, choices
between several distinct classes, or collections of multiple continuous-valued
physical features of an object. We want to design a fusion rule at
$C$ that uses the information from the $A_{m}$ to minimize the Bayes
risk, or in other words, the expected cost of making a wrong decision
\cite{Va96}. In real world problems, the sensor outputs $A_{m}$
are typically deterministic functions of the random object $H$ and
any background noise. We thus make the standard assumption that the
$\{A_{m}\}$ are all conditionally independent given $H$, which is
valid as long as the noise is white and each sensor observes the object
at a slightly different time.\\
\begin{figure}[h]
\centering{}\includegraphics[trim=1.3in 1.2in 1in 1in, clip=true, scale=0.5]{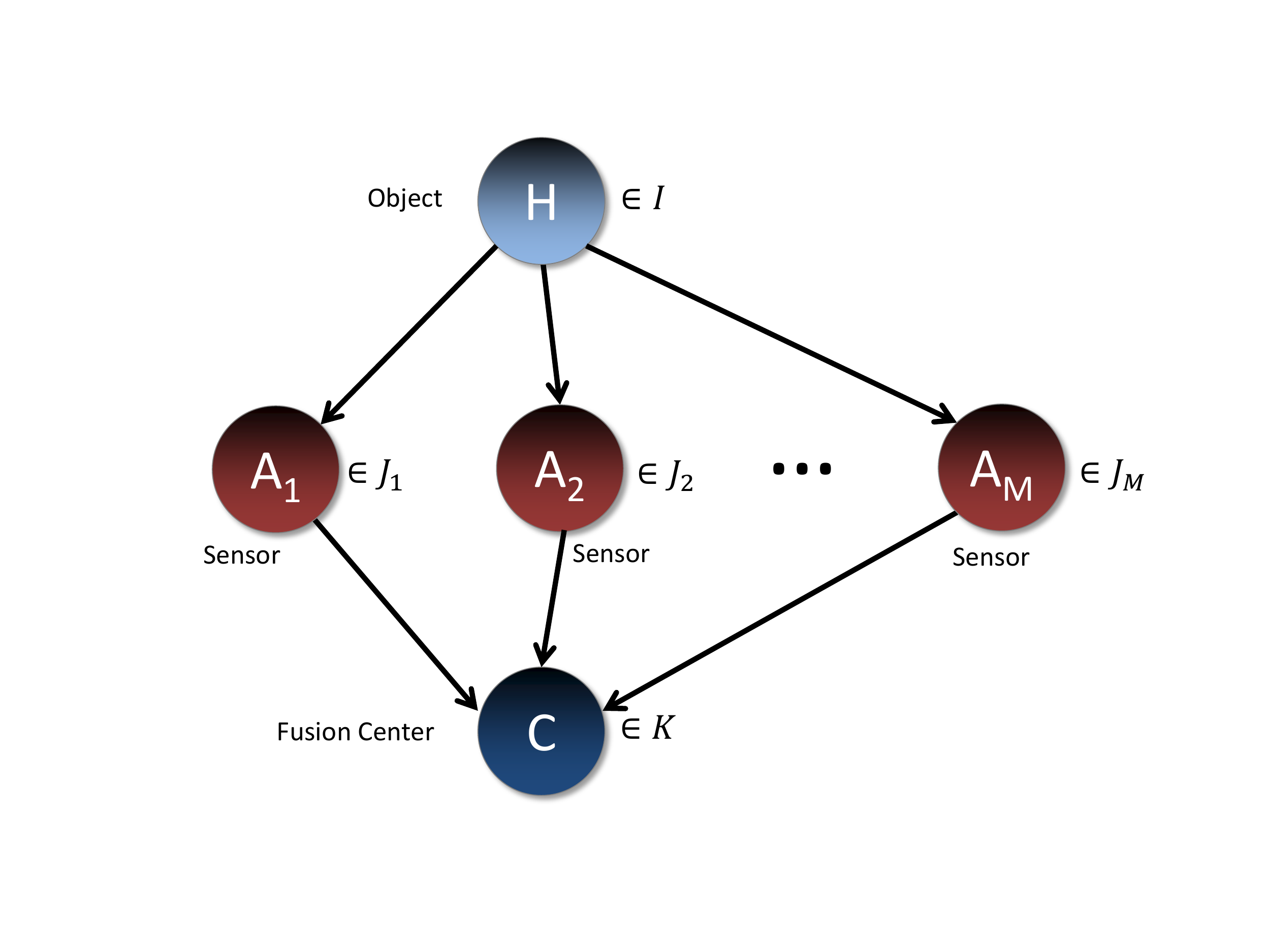}\protect\caption{\label{FigFusion}A probabilistic graphical model of the data fusion
problem.}
\end{figure}

In what follows, we will denote the density and distribution of any
random variable $X$ by $\mathbf{d}_{X}(x)$ and $\mathbf{d}_{X}(x)dx$
respectively, and use the vector notation $\mathbf{A}=\{A_{m}\}_{1\leq m\leq M}$
(where each $A_{m}$ may itself be a vector). For any closed set $K\subset\mathbb{R}$,
we let $C^{0}(K)$, $C_{c}^{0}(K)$ and $C^{\infty}(K)$ respectively
be the space of continuous, bounded functions on $K$, the space of
$C^{0}(K)$ functions with compact support, and the space of $C^{0}(K)$
functions that are smooth (infinitely differentiable with bounded
derivatives on $K$). In order to incorporate deterministic fusion,
we view a density as a first-order generalized function \cite{Ho90,SW71}
(i.e., a Schwartz distribution, but we use the former terminology
to avoid confusion with the meaning of the term in probability theory).
This formulation allows us to consider delta functions and other such
function-like objects. It is equivalent to treating the distribution
as a measure with singular components, but the generalized function
framework is easier to work with in developing the theory below. The
space of first-order generalized functions can be identified with
the dual space of $C_{c}^{0}(K)$ and is denoted $D^{0}(K)$.\\

The fusion problem can now be described as follows. We are given the
densities $\mathbf{d}_{A_{m}|H}(a,h)$ for $1\leq m\le M$, the prior
$\mathbf{d}_{H}(h)$ and a cost function $\mathrm{cost}(c,h)$. We
want to minimize the Bayes risk, or the expected cost of making a
decision $E(\mathrm{cost}(C,H))$. The sensor densities $\mathbf{d}_{A_{m}|H}$
represent either generic statistical models trained from experimental
lab data or mathematical descriptions of the underlying sensor physics,
and the prior $\mathbf{d}_{H}$ comes from domain-specific operational
knowledge. The fusion rule can be expressed as the density $\mathbf{d}_{C|\mathbf{A}}(c,\mathbf{a})$.
We usually want the fusion center $C$ to produce the same type of
information that $H$ represents, so it is reasonable to take the
cost function to be a simple metric between $C$ and $H$ and motivates
the choice $\mathrm{cost}(c,h)=W(c-h)$ for an appropriate increasing
function $W$.\\

We now proceed to establish several results on the optimal deterministic
fusion rule and its performance characteristics. To maintain the clarity
of exposition, we defer the proofs of these theorems to Appendix A.
We first address the classical setting of a quadratic cost $W(x)=\frac{1}{2}x^{2}$.
The optimal fusion rule in this case reduces to the usual posterior
expectation $E(H|C)$ corresponding to a naive Bayes classifier, but
we establish it under the abstract formulation discussed above and
use a variational argument in the proof that will be the foundation
for later results. For the rest of the section, we also assume that
all the conditions of Proposition \ref{ThmDetFusion} are satisfied
unless stated otherwise.\\
\\
\\

\begin{prop}
\noindent \label{ThmDetFusion}Suppose we have a given \textbf{object
space} $I\subset\mathbb{R}$, a collection of \textbf{feature spaces}
$\{J_{m}\}_{1\leq m\leq M}$ with $J_{m}\subset\mathbb{R}^{N_{m}}$,
and a \textbf{decision space} $K$. Suppose $K$ is a single interval
with $I\subset K$. Let the fusion problem be set up as described
above, with the random variables $H$, $A_{m}$ and $C$ respectively
taking on values in $I$, $J_{m}$ and $K$, and let $\mathrm{cost}(c,h)=\frac{1}{2}(c-h)^{2}$.
Assume that $\mathbf{d}_{H}\in D^{0}(I)$, $E(H^{2})<\infty$ and
for every $m$, $\mathbf{d}_{A_{m}|H}\in C^{0}(J_{m}\times I)$ and
$\mathbf{d}_{A_{m}|H}>0$. Then there is an almost everywhere unique,
deterministic, Bayes optimal fusion rule that combines the $\{A_{m}\}$
and produces outputs in $K$. It is given by the posterior expectation
\begin{equation}
f(\mathbf{A})=\frac{\int_{I}h\left(\prod_{m=1}^{M}\mathbf{d}_{A_{m}|H}(A_{m},h)\right)\mathbf{d}_{H}(h)dh}{\int_{I}\left(\prod_{m=1}^{M}\mathbf{d}_{A_{m}|H}(A_{m},h)\right)\mathbf{d}_{H}(h)dh}.\label{FusionRule}
\end{equation}

\end{prop}
The spaces $I$, $J$ and $K$ respectively represent the type of
hidden information we want to estimate or classify between, the type
of information we observe from the sensors and the type of decision
we can make from those observations. We typically think of the objects
and decisions as scalar quantities that may be discrete or continuous,
while the feature spaces are potentially vector quantities (for example,
a camera sensor that outputs images). Since we allow $\mathbf{d}_{H}$
to be a generalized function, the results of Proposition \ref{ThmDetFusion}
cover the case where the object space $I$ is discrete and finite
by taking $\mathbf{d}_{H}$ to be a weighted sum of delta functions.
This is the case in most sensor fusion applications, where there are
a finite and known number of anomaly classes. Note that the deterministic
constraint ((\ref{Cond3}) in the proof) is what allows the fusion
rule to become the posterior expectation (\ref{FusionRule}), and
there can generally be other, random fusion rules (where $\mathbf{d}_{C|\mathbf{A}}$
is a density) that have better performance without this constraint
imposed. Under some additional constraints on the problem, we can
also study the properties of the classification performance of this
fusion rule, given by the density $\mathbf{d}_{C|H}$.
\begin{thm}
\label{ThmPerformance}Assume that $K$ and $\mathbf{J}$ are compact,
$\mathbf{d}_{A_{m}|H}(\cdot,h)\in C^{\infty}(J_{m})$ for all $m$,
and for each point $\mathbf{a}\in\mathbf{J}$ there is at least one
sensor $A_{m'}$ and feature $n$, $1\leq n\leq N_{m'}$, such that
\begin{equation}
\int_{I}\int_{I}\left(\mathbf{d}_{A_{m'}|H}(a_{m'},\tilde{h})\frac{\partial\mathbf{d}_{A_{m'}|H}(a_{m'},h)}{\partial(a_{m'})_{n}}-\mathbf{d}_{A_{m'}|H}(a_{m'},h)\frac{\partial\mathbf{d}_{A_{m'}|H}(a_{m'},\tilde{h})}{\partial(a_{m'})_{n}}\right)h\mathbf{d}_{H}(h)\mathbf{d}_{H}(\tilde{h})dhd\tilde{h}\not=0.\label{BalanceCond}
\end{equation}
Then the classification performance of the fusion rule (\ref{FusionRule}),
given by $\mathbf{d}_{C|H}(\cdot,h)$, is in $C^{\infty}(K)$ for
each $h\in I$.
\end{thm}
The performance function $\mathbf{d}_{C|H}$ of the fusion rule is
a generalization of the classical ``confusion matrix'' that describes
the probabilities of correct and false classification in a binary,
decision-level setting \cite{Va96}. For every object $h\in I$, the
values along the diagonal $\mathbf{d}_{C|H}(h,h)$ are the likelihoods
of the fusion rule giving the correct result, with a delta function
$\mathbf{d}_{C|H}(c,h)=\delta(c-h)$ corresponding to an ideal classifier
(unattainable in a real situation). The condition (\ref{BalanceCond})
holds for simple examples such as those involving Gaussian statistics
(see Section \ref{SecExamples}) but its exact form is not central
to the result, as it can be replaced by a variety of weaker but more
complicated conditions under which the stationary phase argument in
the proof still holds. Under these conditions, Theorem \ref{ThmPerformance}
says that $\mathbf{d}_{C|H}$ is actually a function, instead of merely
a generalized function, and is meaningful at every point. The proof
of Theorem \ref{ThmPerformance} also provides a way to compute the
performance function, by finding
\begin{equation}
\mathbf{d}_{C|H}(c,h)=\int_{\mathbf{J}}\delta(c-f(\mathbf{a}))\left(\prod_{m=1}^{M}\mathbf{d}_{A_{m}|H}(a_{m},h)\right)d\mathbf{a}.\label{LevelCurve}
\end{equation}
This is effectively an integral over only the level sets of the fusion
rule $\{\mathbf{a}\in\mathbf{J}:c=f(\mathbf{a})\}$ in the joint feature
space. In practice, these level sets can be numerically approximated
from the fusion rule as long as either the fusion rule does not have
``flat regions'' where the gradient $\nabla f$ is identically zero,
or the decision space $K$ is discrete. This is usually the simplest
and most efficient approach to computing $\mathbf{d}_{C|H}$, although
various ``smoothed out'' versions of (\ref{LevelCurve}) can be
used instead, such as taking the Fourier transform of (\ref{LevelCurve})
and inverting it, as done in (\ref{Fourier}) in the proof.\\

The computation of the fusion rule (\ref{FusionRule}) itself is simple
and involves an integration over only the (small) object space $I$.
On the other hand, increasing the number of sensors will rapidly increase
the dimension of the integral (\ref{LevelCurve}) for the performance
function. In practice, each sensor generates only a moderate number
of statistics and the individual feature spaces $J_{m}$ are relatively
small (with dimension $N_{m}$ typically at most $10$ or $20$).
This results in the joint density $\mathbf{d}_{\mathbf{A}|H}(\mathbf{a},h)$
having a ``block diagonal'' dependence structure that allows the
calculation of $\mathbf{d}_{C|H}$ to remain computationally tractable.
We will describe a Monte Carlo-based approach to perform this calculation
efficiently in Section \ref{SecNumerical}.\\

We next extend Proposition \ref{ThmDetFusion} to more general decision
spaces $K$ that are not necessarily a single interval. This case
is important for many applications in which $I$ and $K$ are both
finite sets, the so-called binary decision and $M$-ary classification
fusion problems. The case where $I\not\subset K$ is also of interest
in situations where there are several threat substances of interest,
but we ultimately want to make a ``true'' or ``false'' decision
at the fusion center. The posterior expectation (\ref{FusionRule})
is no longer a feasible solution and cannot be applied to this situation,
but we still have the following result.
\begin{thm}
\label{ThmDiscFusion}If the decision space $K$ is a closed set but
otherwise unconstrained, then the Bayes optimal fusion rule is $f^{*}(\mathbf{A})=Qf(\mathbf{A})$,
where $Q$ is the quantization function defined for each $x\in\mathbb{R}$
by choosing any $x'$ from the set $\{\arg\min_{x'\in K}|x-x'|\}$.
This fusion rule is unique almost everywhere.
\end{thm}
The modified fusion rule $f^{*}$ is a generalization of well known
formulas for binary (two element) spaces $I$, $J_{m}$ and $K$ \cite{KZG92,Va96},
and can be computed easily in practice. For example, if $I=K=\{0,1\}$,
then the fusion rule $f$ given by (\ref{FusionRule}) may generally
take on any value in $[0,1]$ and is not feasible, but we can simply
round it to the nearest integer to obtain the actual, optimal fusion
rule $f^{*}$ for our scenario. Note that in the proof of Theorem
\ref{ThmDiscFusion}, the quadratic cost function is crucial and allows
for the cancellation of the third term in (\ref{ThreeTerms}). The
result of Theorem \ref{ThmDiscFusion} will usually not hold for other
costs.\\

We now consider cost functions more general than the quadratic one.
Optimization problems in Bayesian statistics under arbitrary cost
functions usually have no closed-form solutions, and the fusion rule
would in most cases have to be determined numerically (except for
some very specific densities, such as in \cite{Ye94}). However, we
identify a class of prior and sensor densities for which the same
fusion rule (\ref{FusionRule}) turns out to be Bayes optimal for
a much larger class of costs.
\begin{thm}
\label{ThmLpFusion}Suppose that $H$ is sub-Gaussian (i.e., $E(e^{rH{}^{2}})<\infty$
for some $r>0$) and for each $\mathbf{a}\in\mathbf{J}$, whenever
$E(e^{-2\pi izH}|\mathbf{A}=\mathbf{a})=0$ for some $z\in\mathbb{C}$,
$E(e^{2\pi izH}|\mathbf{A}=\mathbf{a})=0$ as well. Let $K$ be a
single interval (possibly unbounded). Then the fusion rule (\ref{FusionRule})
is Bayes optimal for $\mathrm{cost}(c,h)=W(c-h)$, where $W$ is any
entire, even, nonnegative and convex function with $\alpha_{1}(\left|x\right|^{p}-1)\leq\left|W(x)\right|\leq\alpha_{2}(\left|x\right|^{p}+1)$
for some constants \textup{$\alpha_{1},\alpha_{2}>0$} and $p\geq1$.
\end{thm}
The conditions of Theorem \ref{ThmLpFusion} are satisfied, for example,
when the object and sensor statistics are all Gaussian, a case that
is discussed in more detail in Section \ref{SecExamples}. Many other
scenarios where Theorem \ref{ThmLpFusion} holds can also be found,
including cases with finitely supported priors $H$ and other types
of sensor distributions. For example, a prior of the form $\mathbf{d}_{H}(h)=\frac{1}{2M}\sum_{m=1}^{M}(\delta(h+m)+\delta(h-m))$,
corresponding to $I=\{m:m\in\mathbb{Z},1\leq\left|m\right|\leq M\}$,
together with Levy distributed sensor observations $\mathbf{d}_{A_{m}|H}(a_{m},h)=\sqrt{\frac{\left|h\right|}{2\pi a_{m}^{3}}}e^{-\frac{1}{2a_{m}}\left|h\right|}$
can be shown to result in $\widehat{G}(z)$ satisfying the ``symmetric
zeros'' condition of Theorem \ref{ThmLpFusion}.\\

The class of cost functions $W$ addressed by Theorem \ref{ThmLpFusion}
covers a wide variety of interesting cases. It includes functions
of the form $W(x)=\frac{1}{p}x^{p}$ with even exponents $p$, as
well as functions that asymptotically behave like $\left|x\right|^{p}$
for any real $p\geq1$ (e.g., $W(x)=\int_{-\infty}^{\infty}|x-y|^{p}e^{-y^{2}}dy$).
For such cost functions, larger values of $p$ give greater weight
to cases where the object $H$ is something particularly hard to detect
(such as an explosive compound hidden inside an everyday object such
as a mobile phone). Theorem \ref{ThmLpFusion} essentially says that
when detecting such objects is a high priority, (\ref{FusionRule})
is still the best way to reach a decision about them. Note that the
$p=1$ case with $W(x)=|x|$ is not covered, and the optimal decision
rule in this case is the posterior median instead of the posterior
mean (\ref{FusionRule}). The statements on the performance $\mathbf{d}_{C|H}$
from Theorem \ref{ThmPerformance} continue to hold. However, the
Bayes optimal fusion rule is usually no longer unique for non-quadratic
costs, and (\ref{FusionRule}) is one of several possible choices.\\

We finally examine the asymptotic properties of the fusion rule (\ref{FusionRule})
as the number of sensors $M$ increases, with every $N_{m}=1$ to
simplify the notation. For fusion scenarios that are in the class
covered by Theorem \ref{ThmLpFusion}, we can prove much stronger
asymptotic statements that cover not only the Bayes risk but also
the fused classification performance for every possible object.
\begin{thm}
\label{ThmAsymptotic}Let each $J_{m}=\mathbb{R}$ and let $A_{m}$
satisfy $E(A_{m}|H)=H$ and $\mathrm{var}(A_{m}|H)<R$ for some constant
$R$. Suppose $\mathbf{d}_{H}>0$ on $I$ and $K$ is an interval
such that $[\min(h:h\in I)-\epsilon,\max(h:h\in I)+\epsilon]\subset K$
for some $\epsilon>0$. Then as $M\to\infty$, the Bayes risk of the
fusion rule goes to $0$. If in addition $I$ and $K$ are compact
and for each $M$, $\mathbf{A}$ and $H$ satisfy the conditions of
Theorem \ref{ThmLpFusion} and $\mathbf{d}_{C|H}(\cdot,h)$ is a bounded,
even function, then for every $h\in I$, $\mathrm{\mathbf{d}}_{C|H}(\cdot,h)\to\delta(\cdot-h)$
in the weak-$\star$ sense as $M\to\infty$.
\end{thm}
The proof of Theorem \ref{ThmAsymptotic} shows that for a large number
of sensors $M$, the Bayes risk of the optimal fusion rule is bounded
above by the sample mean of the sensor outputs. The quadratic Bayes
risk of the optimal fusion rule decays at least as fast as (and possibly
faster than) the $O(\frac{1}{M})$ rate that the sample mean achieves,
but it is usually quite different for fixed and realistic values of
$M$. The conditions in Theorem \ref{ThmLpFusion} are the main things
that enable the pointwise-like convergence result $\mathbf{d}_{C|H}(\cdot,h)\to\delta(\cdot-h)$,
and the additional constraints of Theorem \ref{ThmAsymptotic} (such
as compact $I$ and $K$) simplify the proof but can be relaxed. If
the object and feature distributions are not in the class covered
by Theorem \ref{ThmLpFusion}, the performance $\mathbf{d}_{C|H}$
of $f$ may not necessarily approach ``perfect classification''
(a delta function) and is only guaranteed to do so on average, in
the sense of the Bayes risk under a quadratic cost going to zero.

\section{Example data fusion scenarios\label{SecExamples}}

In this section, we consider a series of examples applying the theory
from Section \ref{SecMain} to concrete data fusion and performance
analysis problems. The simplest situation is when the sensor and object
statistics are all Gaussian, and is one of only a few cases where
the performance function of the fusion rule can be calculated symbolically.
We study this case in detail.
\begin{prop}
\label{ThmGauss}Let $I=K=\mathbb{R}$ and $J=J_{m}=\mathbb{R}^{M/2}$
for some even $M$. Suppose the object $H$ is Gaussian with mean
$0$ and variance $1$. Suppose the sensors $A$ and $B$ observe
$H$ and respectively output collections of features $\{A_{m}\}_{1\leq m\leq M/2}$
and $\{B_{m}\}_{1\leq m\leq M/2}$, where each $A_{m}|H$ and $B_{m}|H$
is Gaussian with mean $uH$ and variance $v$ for some fixed parameters
$u$ and $v>0$. Then the optimal fusion rule under a quadratic cost
(\ref{FusionRule}) is
\begin{equation}
f((\mathbf{A},\mathbf{B}))=\frac{u}{Mu^{2}+v}\left(\sum_{m=1}^{M/2}A_{m}+\sum_{m=1}^{M/2}B_{m}\right),\label{GaussFusion}
\end{equation}
and its performance and Bayes risk are given by
\begin{eqnarray*}
\mathbf{d}_{C|H}(c,h) & = & \frac{Mu^{2}+v}{u\sqrt{2\pi Mv}}e^{-\frac{\left(Mu^{2}(c-h)+cv\right)^{2}}{2Mu^{2}v}}\\
E((C-H)^{2}) & = & \frac{v}{Mu^{2}+v}.
\end{eqnarray*}
Furthermore, (\ref{GaussFusion}) is optimal for all costs $W$ satisfying
the conditions of Theorem \ref{ThmLpFusion}.
\end{prop}
The situation described by Proposition \ref{ThmGauss} can be interpreted
in the following way. The object space represents a degree of belief
between the certain presence ($H=\infty$) and certain absence $(H=-\infty)$
of an object of interest. Each of the features picked up by the sensors
$A$ and $B$ provide some information on $H$, but with a known level
of uncertainty $v$. The fusion rule (\ref{GaussFusion}) combines
the features in a way that best matches the resulting (soft) decision
with the original belief, using the available information from the
sensors. The division of the $M$ features into two sensors $A$ and
$B$ is arbitrary and they can equivalently be combined into a single
sensor, but this formulation is useful in establishing Corollary \ref{CorGaussSep}
below.\\

Proposition \ref{ThmGauss} shows that the average performance of
the fusion rule (as measured by the quadratic Bayes risk) is roughly
inversely proportional to the number of features being fused. Note
that the optimal fusion rule (\ref{GaussFusion}) is similar to but
different from the sample mean of all the features $A_{m}$ and $B_{m}$,
and that its performance is skewed by the prior on $H$ (see Figure
\ref{FigExGauss}). For a large number of features, it is easy to
verify that as $M\to\infty$, the fusion rule (\ref{GaussFusion})
approaches the actual sample mean of the $A_{m}$ and $B_{m}$ and
the performance satisfies $\mathbf{d}_{C|H}(\cdot,h)\to\delta(\cdot-h)$
in the weak-$\star$ sense for each $h$, in line with Theorem \ref{ThmAsymptotic}.\\

\begin{figure}[h]
\centering{}\includegraphics[trim=0.2in 0in 0.5in 0in, clip=true, scale=0.6]{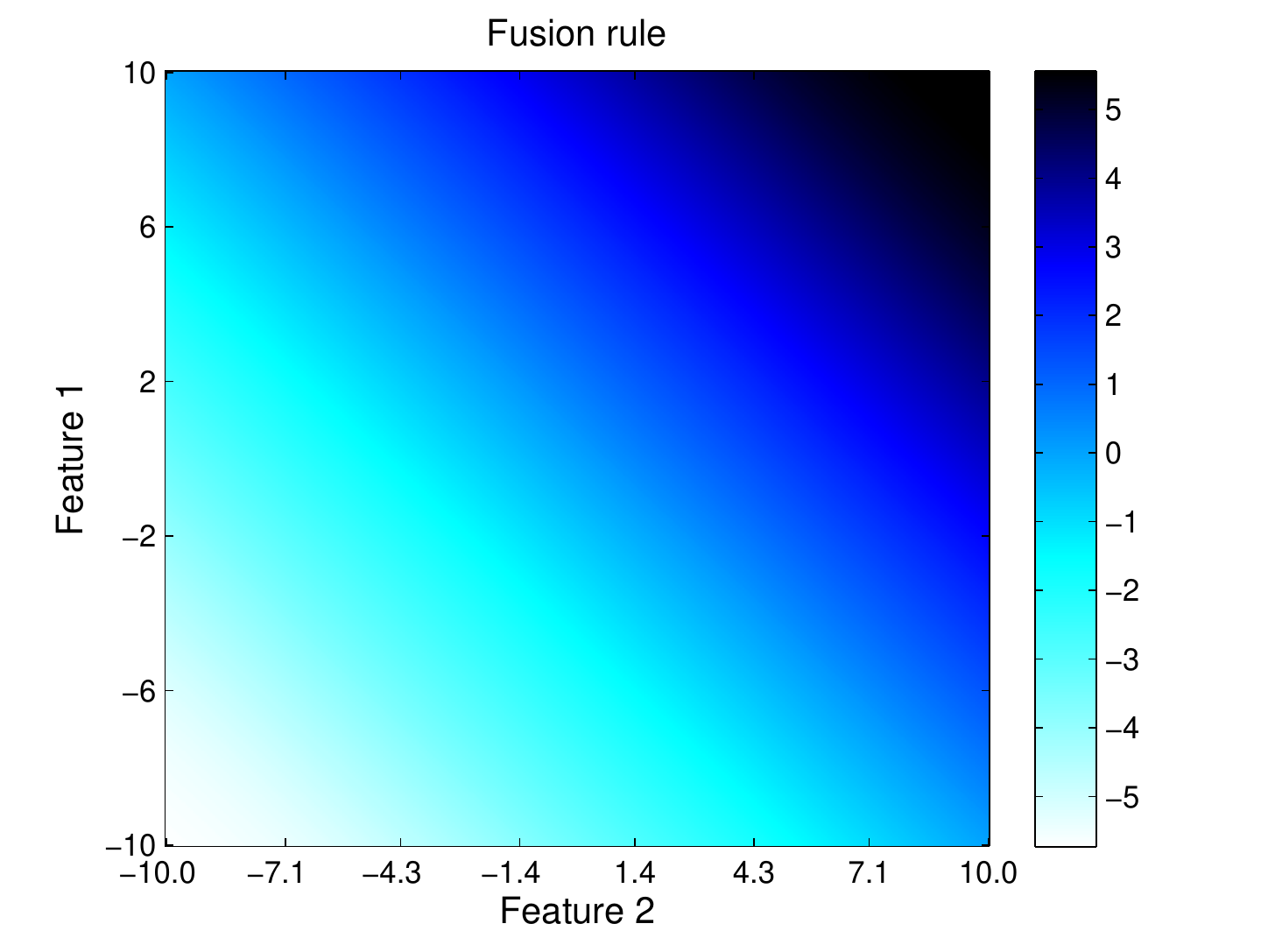}\includegraphics[trim=0.2in 0in 0.5in 0in, clip=true, scale=0.6]{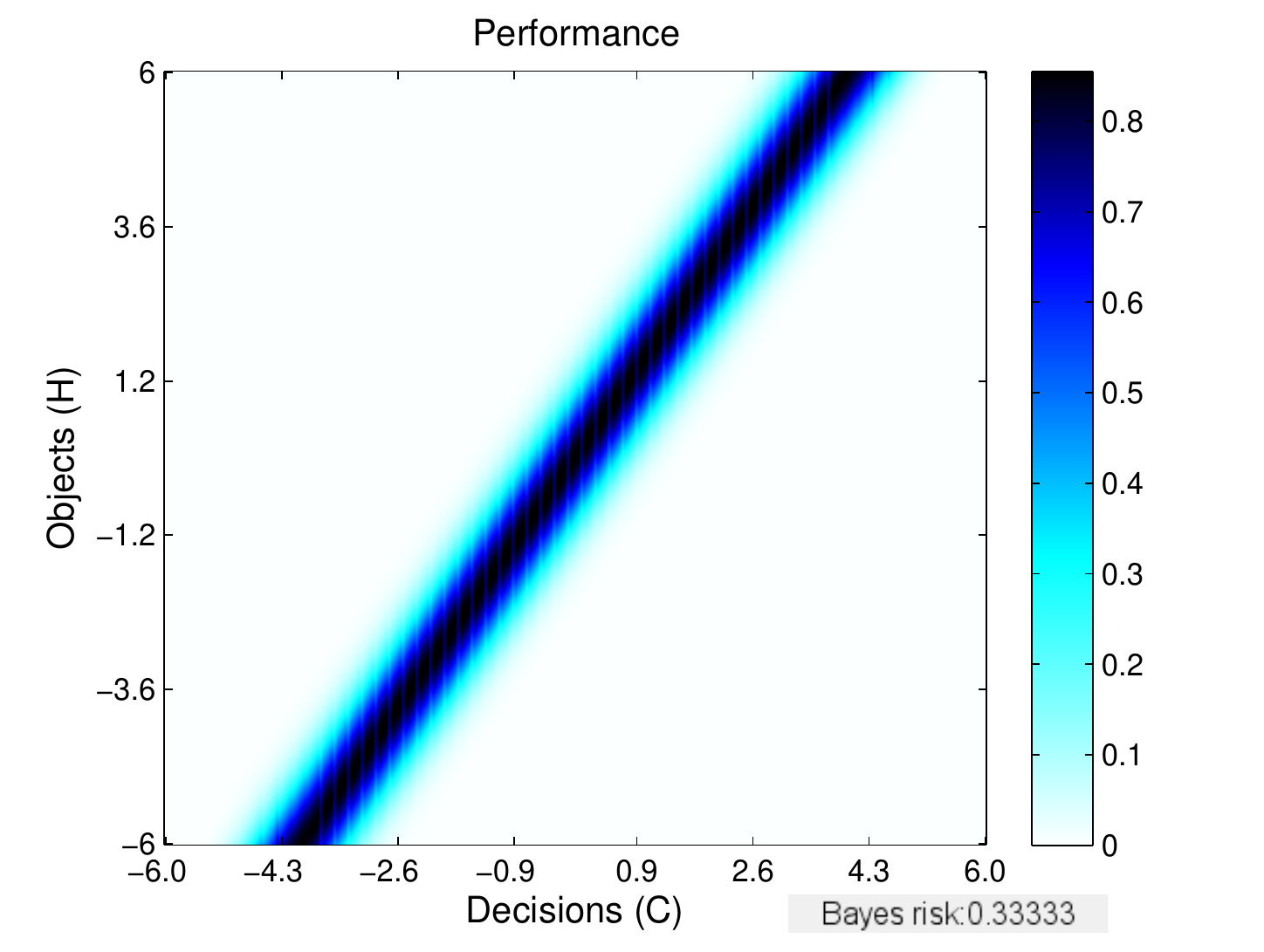}\protect\caption{\label{FigExGauss}Fusion configuration from Proposition \ref{ThmGauss}
with $u=v=1$ and $M=2$.}
\end{figure}

Proposition \ref{ThmGauss} also allows us to compare the performance
of directly fusing all features for the entire system, as opposed
to having each sensor combine its own features internally according
to (\ref{GaussFusion}) and then fusing the resulting sensor outputs
using (\ref{GaussFusion}) again (see Figure \ref{FigDistFusion}
(a-b)). The latter corresponds to a typical approach taken in many
real-world sensor fusion systems and is a version of distributed fusion
with person-by-person optimization (PBPO), where the decision rule
at each fusion center is chosen using only the properties of its own
inputs and outputs, without regard to the rest of the graphical model
(see \cite{BKP97,Du07,Va96,VV97}). Distributed fusion networks are
common in many applications with large numbers of sensors due to bandwidth
or other communication constraints between the sensors. This typically
incurs a loss in the overall system performance, but in the special
case when the statistics are all Gaussian, the performance is unaffected.
\begin{cor}
\label{CorGaussSep}Let $I$, $J$, $K$, $H$, $A$ and $B$ be as
given in Proposition \ref{ThmGauss}, with $K^{\star}=K$ and $u=v=1$,
so the Bayes risk of the fusion configuration in Proposition \ref{ThmGauss}
is $\frac{1}{M+1}$. Now suppose $A$ and $B$ each have internal
fusion centers that respectively reduce $\{A_{m}\}_{1\leq m\leq M/2}\in J$
to a decision $A^{\star}\in K^{\star}$ and $\{B_{m}\}_{1\leq m\leq M/2}\in J$
to a decision $B^{\star}\in K^{\star}$ using locally Bayes optimal
decisions. Let $A^{\star}$ and $B^{\star}$ be combined at the system
fusion center $C$ to produce a locally optimal output in $K$, as
in Figure \ref{FigDistFusion} (b). Then the Bayes risk of the entire
fusion system is still $\frac{1}{M+1}.$
\end{cor}
The result of Corollary \ref{CorGaussSep} is possible only because
of the simple form of the fusion rule (\ref{GaussFusion}). The Bayes
risk in a distributed fusion model like this will generally be lower
in other cases due to a loss of information at $A^{\star}$ and $B^{\star}$.
An example of this is discussed below.\\

Another example similar to Proposition \ref{ThmGauss} can be considered,
using other types of sensor statistics.
\begin{prop}
\label{ThmExpo} Let $I=K=J_{m}=[0,\infty)$. Suppose the object $H$
is exponentially distributed with rate parameter $1$ (i.e., $E(H)=1$),
and there are $M$ sensors $A_{m}$ with exponentially distributed
observations all having rate parameter $H$. Then the optimal fusion
rule is $f(\mathbf{A})=\frac{M+1}{\sum_{m=1}^{M}A_{m}+1}$. Its performance
and Bayes risk are given by
\begin{eqnarray*}
\mathbf{d}_{C|H}(c,h) & = & \frac{(M+1)h^{M+1}}{M!c^{2}}\left(\frac{M+1}{c}-1\right)^{M}e^{-h(\frac{M+1}{c}-1)},\quad c\in(0,M+1]\\
E((C-H)^{2}) & = & \frac{2}{M+2}.
\end{eqnarray*}

\end{prop}
Proposition \ref{ThmExpo} illustrates a variety of different behavior
from the Gaussian cases. The fusion rule is now quite different from
the sample mean and only takes on values between $0$ and $M+1$,
even though the decision space is the entire positive real axis. This
indicates that any fusion rule that produces decisions $f(\mathbf{A})>M+1$
would compromise the good performance of the decisions for $0<f(\mathbf{A})\leq M+1$,
to the extent that the overall Bayes risk of the system increases.
The performance function is only meaningful for such values of $c$,
but if we define it to be identically zero for $c\geq M+1$, then
it still converges to a delta function in the weak-$\star$ sense
as $M\to\infty$.

\begin{figure}[h]
\centering{}\includegraphics[trim=0in 0.5in 0in 0.5in, clip=true, scale=0.35]{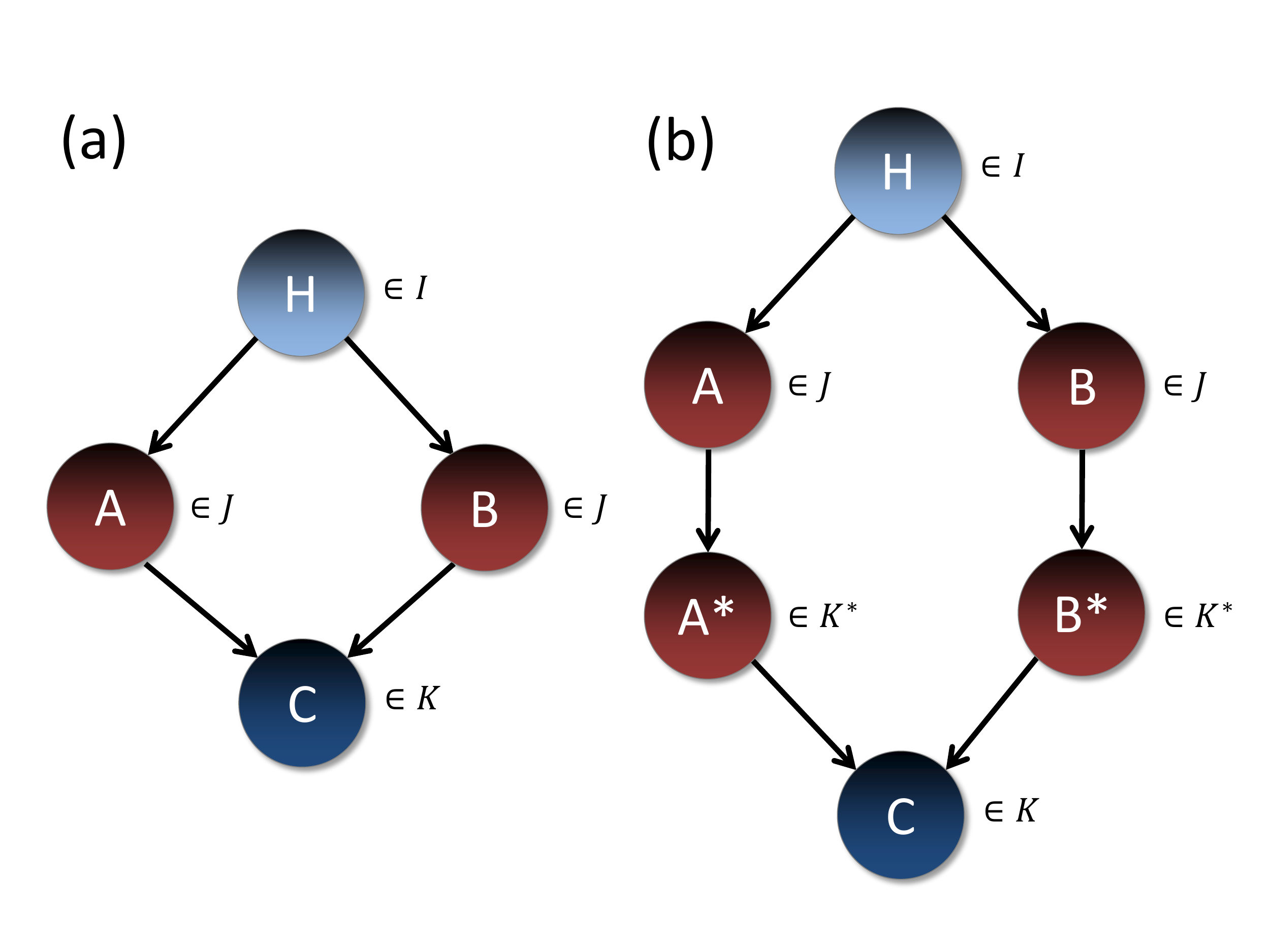}\protect\caption{\label{FigDistFusion}Graphical models of fusion configurations for
(a) centralized fusion of features, and (b) distributed PBPO fusion
of features, i.e., sensor-level fusion of features followed by system-level
fusion of the resulting decisions.}
\end{figure}

We next look at a simple example with two sensors, each producing
one feature, and a discrete, finite object space. It is generally
no longer possible to study the fusion performance symbolically, but
we can compute it numerically using (\ref{LevelCurve}). Let $I=\{0,1,2,3\}$
and $J_{1}=J_{2}=\mathbb{R}$. For the sensors $A$ and $B$, let
$A|H$ and $B|H$ be Gaussian variables for each $H\in I$, with $E(A|H)=E(B|H)=\{0,1,2,3\}$,
$\mathrm{var}(A|H)=\{1.7,0.4,3,1\}$ and $\mathrm{var}(B|H)=\{0.5,2,0.7,2\}$.
We consider the cases $K=\{0,1,2,3\}$ and $K=[0,3]$, corresponding
to hard or soft decisions at the fusion center. The resulting fusion
rules and their performance functions are shown in Figure \ref{FigExMixed}.
The hard decision fusion rule contains small ``islands'' surrounded
by larger regions where different decisions are made, and this behavior
is typical of problems where $I$ and $K$ are discrete but the $J_{m}$
are continuous. The soft decision scenario has an improved Bayes risk
(0.35536) over the hard decision case (0.43775), reflecting the fact
that more information is preserved about the object at the fusion
center. However, some of the classification performance of the object
$H=1$ was traded off for improved performance with $H=0$.\\

These scenarios can also be placed in the context of Corollary \ref{CorGaussSep}
and the distributed fusion model in Figure \ref{FigDistFusion} (b).
If we take $H$, $A$ and $B$ as above and set $K^{\star}=\mathbb{R}$
and $K=[0,3]$, then the fusion centers at $A^{\star}$ and $B^{\star}$
are simply identity mappings that pass the outputs of $A$ and $B$
into $C$, so the Bayes risk at $C$ is the same as the soft decision
scenario above. On the other hand, if we take $K^{\star}=\{0,1,2,3\}$,
so that the outputs of $A$ and $B$ are first reduced to hard decisions
before the fusion at $C$, then the Bayes risk at $C$ turns out to
be 0.57862, even though the final decision space $K$ is unchanged.
This shows how distributed fusion can hurt performance when enough
information is lost at $A^{\star}$ and $B^{\star}$, as opposed to
fusing $A$ and $B$ directly at $C$.\\

\begin{figure}[h]
\centering{}\includegraphics[trim=0.2in 0in 0.5in 0in, clip=true, scale=0.6]{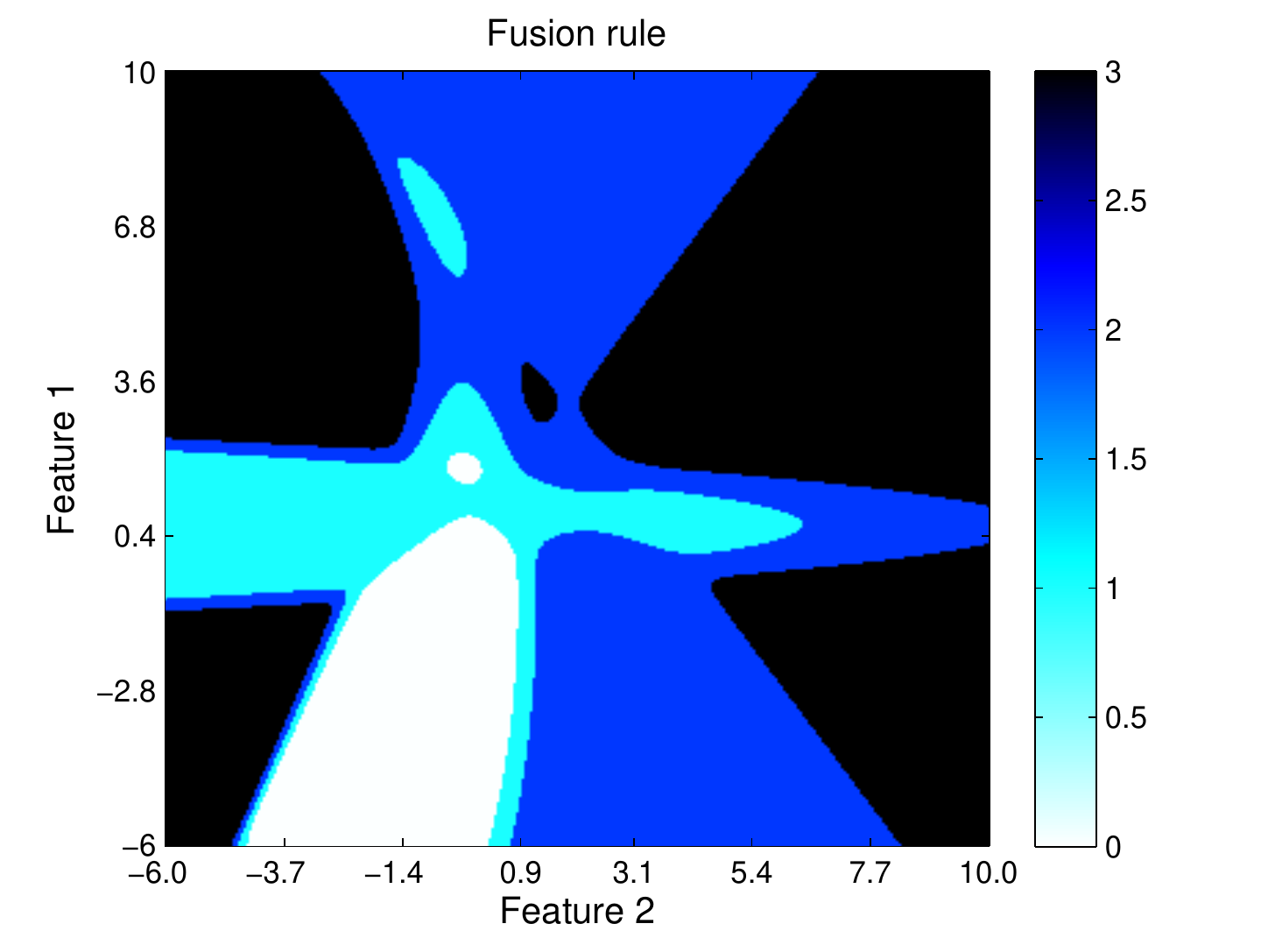}\includegraphics[trim=0.2in 0in 0.5in 0in, clip=true, scale=0.6]{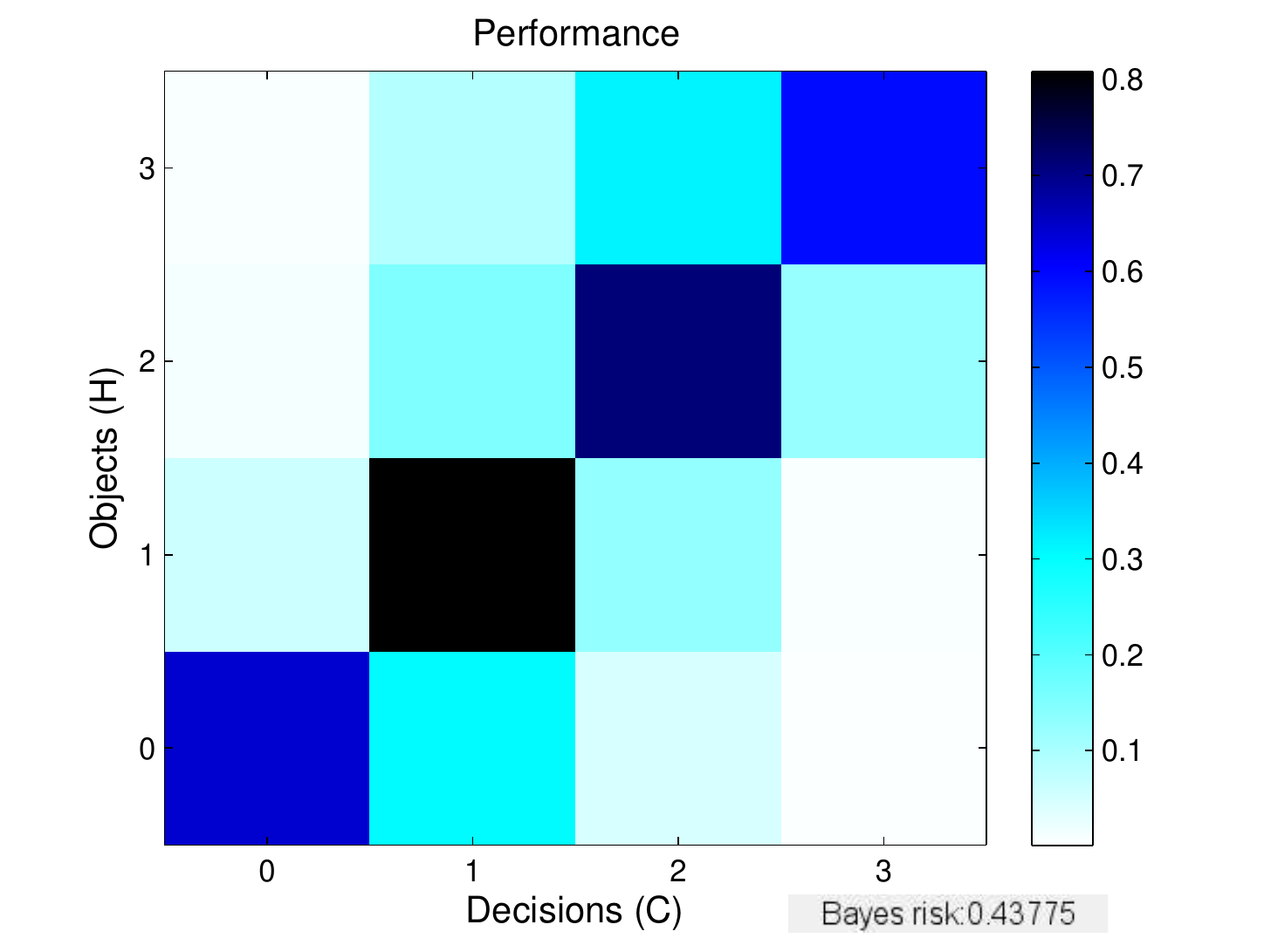}\\
\includegraphics[trim=0.2in 0in 0.5in 0in, clip=true, scale=0.6]{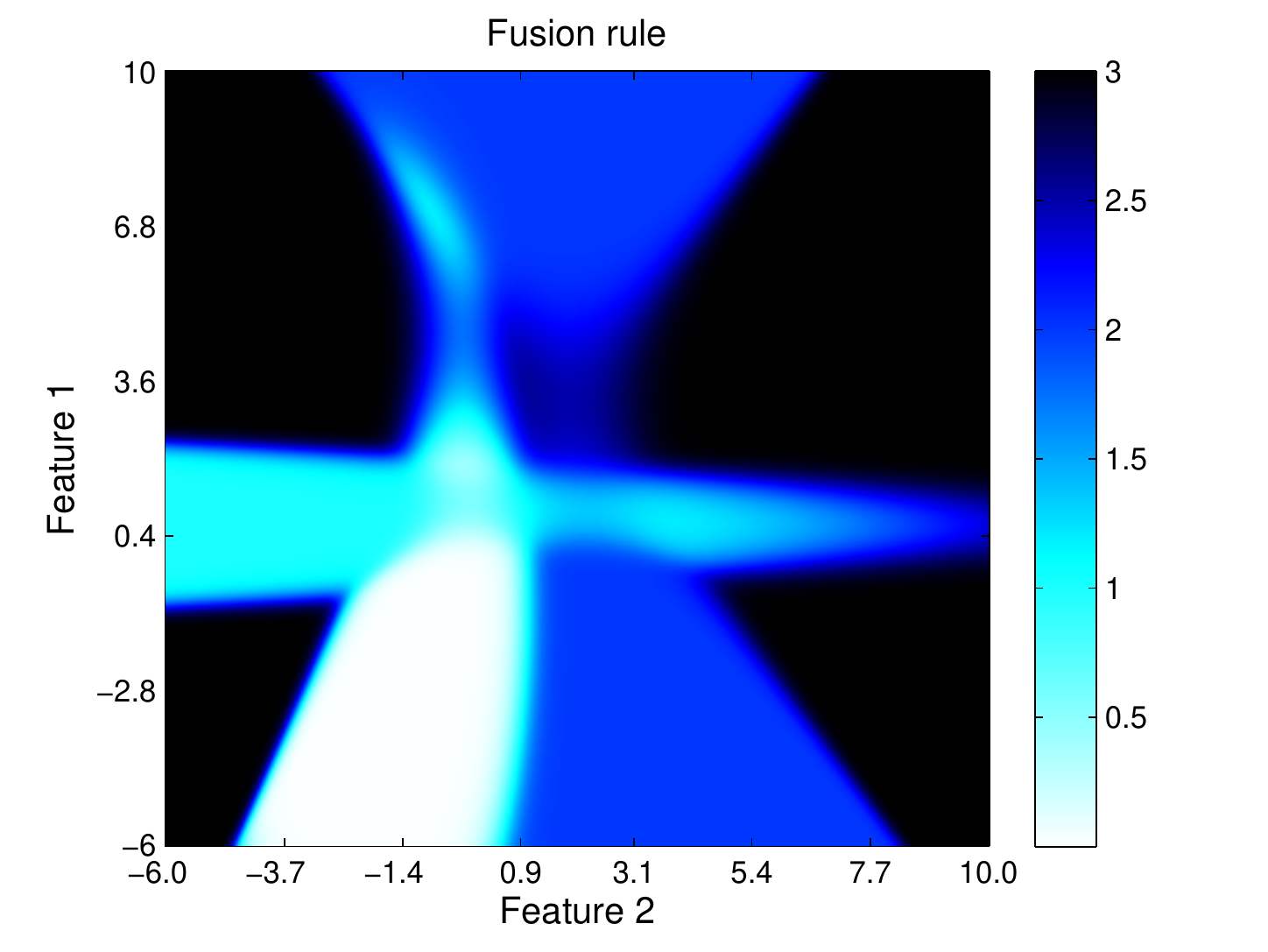}\includegraphics[trim=0.2in 0in 0.5in 0in, clip=true, scale=0.6]{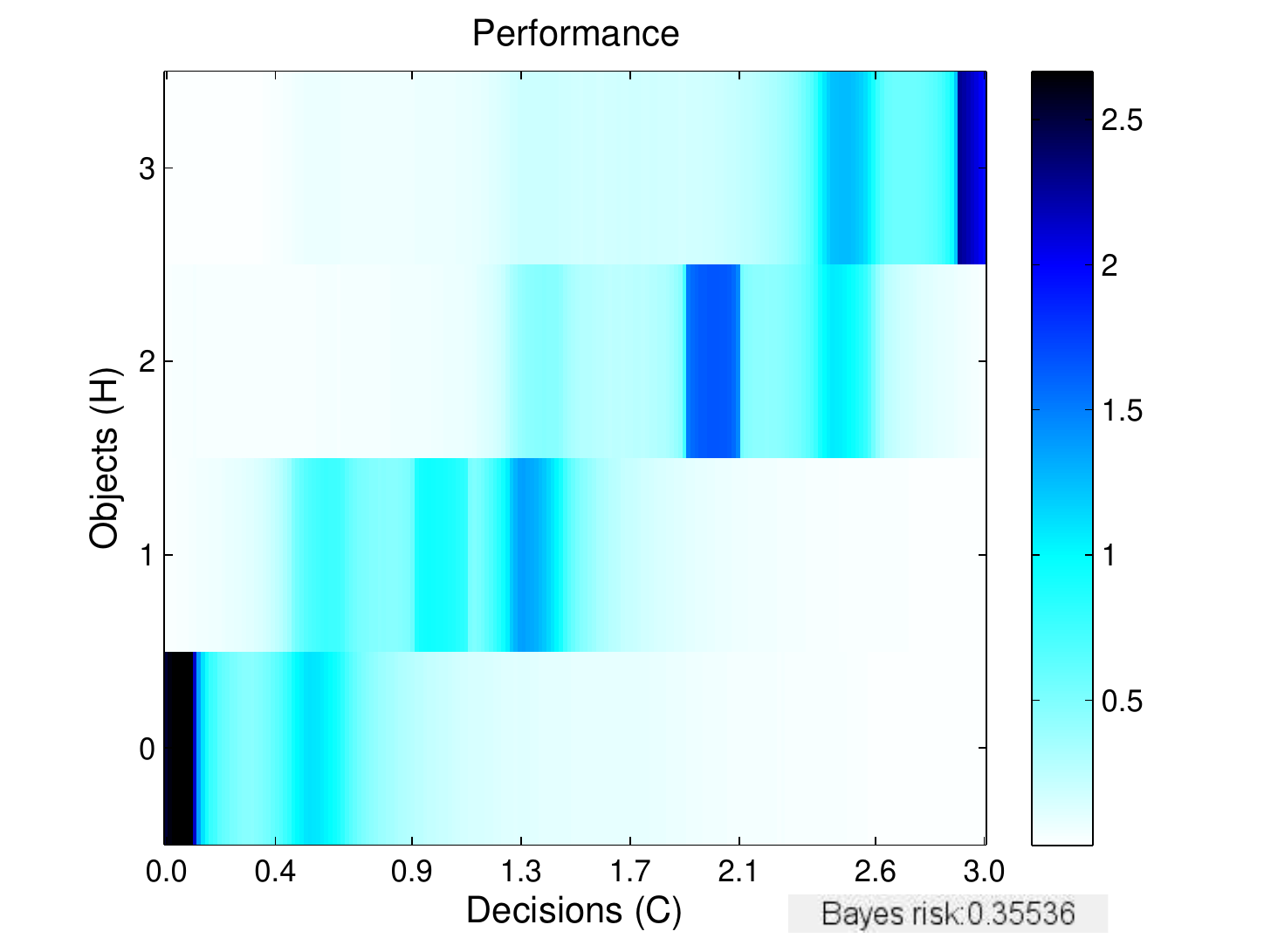}\protect\caption{\label{FigExMixed}Fusion rules and their performance with a discrete
object space, Gaussian sensor outputs and a discrete (top) and continuous
(bottom) decision space.}
\end{figure}

Another, similar type of scenario can be considered with discrete
features and shows how random fusion rules can enter the picture.
Let $I=K=\{1,2\}$ and $J_{1}=J_{2}=\{0,1,2,3,...\}$. Let $P(H=1)=P(H=2)=\frac{1}{2}$
and $A|H$ and $B|H$ be Poisson variables with rate parameter $H$.
By Theorem \ref{ThmDiscFusion}, the deterministic fusion rule for
a quadratic cost is given by taking $f(A,B)=1$ when $A+B\leq2$ and
$f(A,B)=2$ otherwise. In other words, there are only six $(A,B)$
pairs with $f(A,B)=1$. The classification performance of this rule
can be found explicitly, with the false positive rate $P(H=1|C=2)=1-5e^{-2}$
and the miss rate $P(H=2|C=1)=13e^{-4}$. Now the only way to improve
the miss rate is by mapping one of the six $(A,B)$ values to $f(A,B)=2$
instead, but there are a countable number of fusion rules that do
this and their performances can only take on specific values. If we
change the cost to be such that the false positive rate can be at
most $1-(5-\epsilon)e^{-2}$ for some small $\epsilon>0$, then there
are random fusion rules with lower miss rates than any of the deterministic
rules. For example, the random fusion rule $g$ that takes $g(1,1)=2$
with probability $\epsilon$ and $1$ otherwise, and $g(A,B)=f(A,B)$
for all other $(A,B)$, is a (non-unique) optimal choice with $P(H=1|C=2)=1-(5-\epsilon)e^{-2}$
and $P(H=2|C=1)=(13-4\epsilon)e^{-4}$. This situation is typical
when the feature and/or decision spaces are discrete, with a random
fusion rule having more ``wiggle room'' to achieve specific classification
probabilities, and is essentially a special case of the classical
Neyman-Pearson lemma for likelihood ratio tests (\cite{Po94}, p.
23).\\

We finally consider one more numerical example with more complex features
that have mixture distributions. Let $I=K=[0,4]$ and $J_{1}=J_{2}=[0,5]$.
Suppose that $H$ has a Gaussian prior with mean $0$ and variance
$1$ as in Proposition \ref{ThmGauss}, but $A$ is an exponential-uniform
mixture with the form $\mathbf{d}_{A|H}(a,h)=\frac{h}{2}e^{-ha}+\frac{1}{2h}\chi_{[0,h]}(a)$,
where $\chi$ is the indicator function, and $B$ follows the Gaussian
mixture distribution $\mathbf{d}_{B|H}(b,h)=\frac{1}{2}(\frac{\pi}{50})^{-1/2}e^{-50(b-h)^{2}}+$$\frac{1}{2}(18\pi)^{-1/2}he^{-\frac{1}{18}h^{2}(b-0.7)^{2}}$.
The optimal fusion rule and its performance function are shown in
Figure \ref{FigExMixture}. Note that the optimal fusion rule never
outputs decisions greater than about $2.5$, which is reflected in
its performance and is similar to the scenario in Proposition \ref{ThmExpo}.

\begin{figure}[H]
\centering{}\includegraphics[trim=0.2in 0in 0.5in 0in, clip=true, scale=0.6]{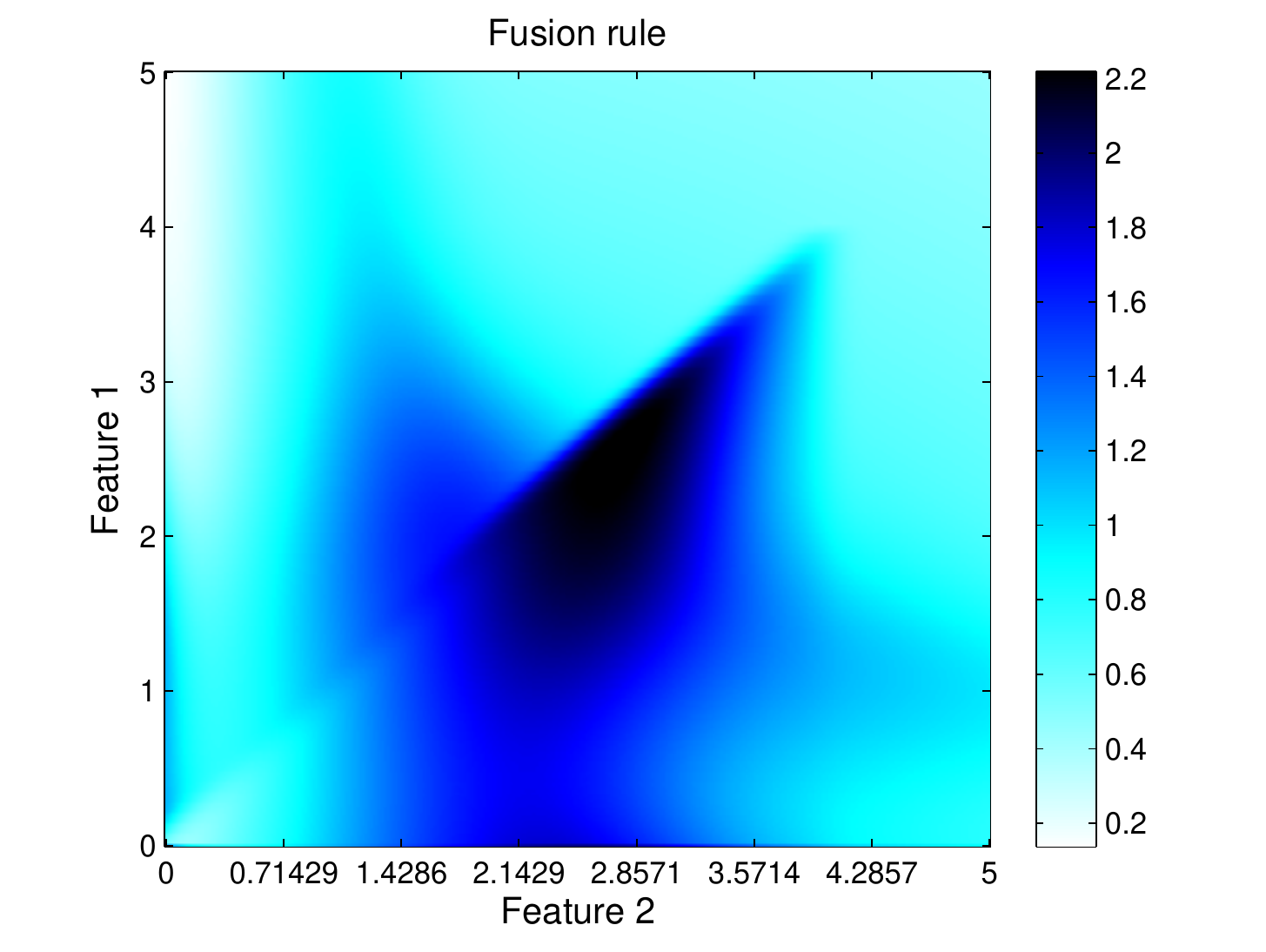}\includegraphics[trim=0.1in 0in 0.5in 0in, clip=true, scale=0.6]{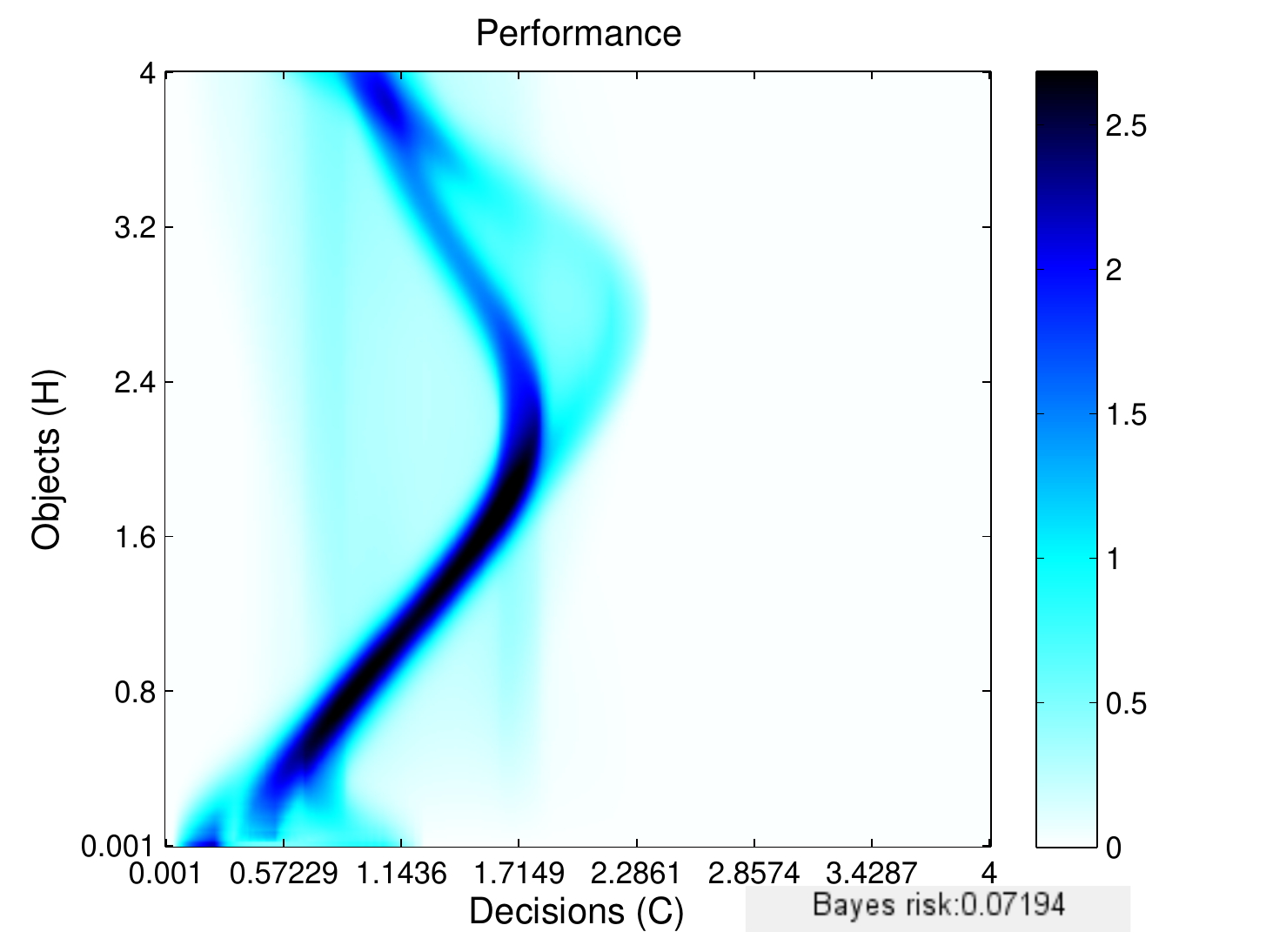}\protect\caption{\label{FigExMixture}Fusion configuration with mixture distributions.}
\end{figure}

\section{Computational methods\label{SecNumerical}}

We discuss some approaches to efficiently compute the performance
$\mathrm{\mathbf{d}}_{C|H}$ and Bayes risk $E(\mathrm{cost}(C,H))$
of the fusion rule (\ref{FusionRuleProof}) in realistic scenarios
with a large number of sensors or features. We assume that the densities
$\mathbf{d}_{A_{m}|H}$ are available in a symbolic (but not necessarily
closed) form, coming from either a physics-based model or a kernel
density estimate or other statistical model fitted to experimental
data. The $A_{m}|H$ can be discrete or continuous variables and have
different dimensions $N_{m}$, depending on what kinds of information
each sensor puts out, but they are assumed to be well localized in
their respective feature spaces. Having a symbolic expression for
$\mathbf{d}_{A_{m}|H}$ as opposed to a tabulation on a discrete grid
(for continuous variables) allows us to sample at points anywhere
in $J_{m}$, which is essential for the use of randomized integration
methods that scale efficiently in the total number of dimensions $\mathbf{N}=\sum_{m=1}^{M}N_{m}$.
For sensors that produce several different features ($N_{m}\geq1$),
$\mathbf{d}_{A_{m}|H}$ is often not specified explicitly but is given
in terms of a probabilistic graphical model \cite{Jo04} that describes
the dependencies among the individual features $A_{m_{n}}$ and any
intermediate (nuisance) variables.\\

In typical fusion scenarios in practice, it is common to combine several
hundred features at once, and determining the performance $\mathrm{\mathbf{d}}_{C|H}$
leads to a high dimensional integral in (\ref{LevelCurve}) that is
intractable by conventional lattice-based approaches. However, one
key property of statistical inference problems such as this is that
the sensor densities $\mathbf{d}_{A_{m}|H}$ are all nonnegative,
which means that the integral (\ref{LevelCurve}) involves no cancellation
and the largest contribution comes from around the local maxima of
the $\mathbf{d}_{A_{m}|H}$. We outline a Monte Carlo importance sampling
approach that is motivated by these observations. We collect samples
$\{(\mathbf{a}_{l},h_{l})\}_{1\leq l\leq L}$ drawn from the proposal
distribution $\left(\prod_{m=1}^{M}\mathbf{d}_{A_{m}|H}(a_{m},h)da_{m}\right)\mathbf{d}_{H'}(h)dh$,
which prioritizes points around the maxima of (\ref{LevelCurve})
and reduces the variance in the resulting estimates \cite{GS90}.
$H'$ is a variable with either the same distribution as $H$ or the
uniform distribution on $I$. This reflects an accuracy tradeoff between
finding the Bayes risk and the performance, with the former choice
more efficient for calculating the Bayes risk and the latter preferable
for finding the performance. For a given number of samples, the former
choice will pick the points in $H$ that contribute the most to the
Bayes risk, but may leave a large portion of the domain $I\times K$
uncovered by the samples and result in an inaccurate performance function.
On the other hand, sampling $H$ according to a uniform distribution
on $I$ will distribute the points evenly across $I\times K$, even
though only a few may add significantly to the Bayes risk. Note that
finding the performance $\mathrm{\mathbf{d}}_{C|H}$ corresponds to
a ``downstream'' calculation in the probabilistic graphical model
in Figure \ref{FigFusion}, where samples are generated at $H$ and
propagate downward through the sensors $A_{m}$ into $C$. This is
in contrast to the more conventional task of doing posterior inference
on data, which amounts to finding $\mathrm{\mathbf{d}}_{H|\mathbf{A}}$
and is an ``upstream'' calculation that involves the Bayes formula.\\

The sensor densities $\mathbf{d}_{A_{m}|H}$ can be sampled from using
a variety of approaches, depending on how each one is specified. For
graphical models, standard Markov Chain Monte Carlo (MCMC) methods
such as Gibbs sampling and its variants (see \cite{Bi06} for details)
allow us to obtain samples from a high-dimensional joint density that
would be impossible to sample from directly, unless it has some special
structure. However, one such case arises frequently in practice, where
each $A_{m}|H$ is jointly Gaussian. Many standard types of radar
and acoustic sensors for explosive detection collect measurements
such as the signal's return time or its power at specific frequencies.
These measurements are typically formed by averaging a large number
of consecutive ``looks'' to smooth out the effects of noise, which
results in each $A_{m}|H$ being approximately an $\mathbb{R}^{N_{m}}$-valued
Gaussian variable, and the resulting joint distribution is easy to
sample from directly. The individual components of $A_{m}|H$ are
usually not independent, and may represent measurements such as the
signal intensity at different frequencies, which are all influenced
by an explosive substance with a given spectral profile. However,
we can simply take $N_{m}$ independent Gaussian samples $\mathbf{G}=\{G_{n}\}_{1\leq n\leq N_{m}}$
(using the Gaussian quantile function) with mean $0$ and variance
$1$, and ``color'' them appropriately by taking $\mu+\mathbf{V}\mathbf{G}$,
where $\mu=\mathrm{mean}(A_{m}|H)$ and $\mathbf{V}\mathbf{V}^{T}=\mathrm{cov}(A_{m}|H)$
is the Cholesky decomposition.\\

Once a sequence of samples $\{(\mathbf{a}_{l},h_{l})\}$ has been
obtained, we can use the fact that for each $\mathbf{a}$ there is
exactly one $c\in K$ with $c=f(\mathbf{a})$, so the delta function
in the performance integral (\ref{LevelCurve}) never has to be computed
or approximated explicitly. Instead, we discretize the decision space
$K$, and for each sample $(\mathbf{a}_{l},h_{l})$, we find the closest
$c$ in the discretized space and add $\prod_{m=1}^{M}\mathbf{d}_{A_{m}|H}((a_{l})_{m},h)$
to the sum corresponding to that $(c,h)$ pair, effectively producing
a weighted histogram of $\{f(\mathbf{a}_{l}|h_{l})\}$ to determine
$\mathrm{\mathbf{d}}_{C|H}$. The Bayes risk may be found from the
same samples $\{(\mathbf{a}_{l},h_{l})\}$ directly, or from the performance
by computing $\int_{I\times K}\mathrm{cost}(c,h)\mathrm{\mathbf{d}}_{C|H}(c,h)\mathrm{\mathbf{d}}_{H}(h)dcdh$.
Standard confidence bounds on the estimated Bayes risk and performance
function can be found from the central limit theorem, which also holds
for dependent variables with sufficiently good ergodicity or mixing
properties (as is the case with some MCMC sampling patterns \cite{JG12}).
For example, let $B_{L}(W)$ be the Monte Carlo estimate of the Bayes
risk $B(W)$ under the cost function $W$ with $L$ samples taken
from the proposal distribution with $H'=H$. Then for any confidence
level $0<R<1$, as $L\to\infty$,
\[
P\left(\left|B_{L}(W)-B(W)\right|<\left(\frac{2}{L}\int_{I\times\mathbf{J}}\left(W(f(\mathbf{a})-h)-B(W)\right)^{2}G(\mathbf{a},h)d\mathbf{a}\, dh\right)^{1/2}\mathrm{erf}^{-1}(R)\right)\to R,
\]
and for fusion problems and costs covered by Theorem \ref{ThmLpFusion},
this implies that for sufficiently large $L$,
\[
P\left(\left|B_{L}(W)-B(W)\right|<\sqrt{\frac{2}{L}}\left(B_{L}(W^{2})^{1/2}+B_{L}(W)\right)\mathrm{erf}^{-1}(R)\right)\geq R.
\]

We use the approach discussed here to study a larger version of the
fusion scenario considered in Proposition \ref{ThmGauss}. In Figure
\ref{FigExComputational}, we take $M=300$ total features in Proposition
\ref{ThmGauss} with $60000$ i.i.d. points sampled from the proposal
distribution with a uniform $H'$. The performance can be compared
with Figure \ref{FigExGauss} and, as expected, is much more concentrated
along the main diagonal, with a correspondingly lower Bayes risk.
Note that as we increase the number of features $M$, the quadratic
Bayes risk decays at least as fast as $O(\frac{1}{M})$ by Theorem
\ref{ThmAsymptotic}, so we need roughly $L=O(M)$ sampling points
to achieve a fixed relative error in the calculation.

\begin{figure}[h]
\centering{}\includegraphics[clip,scale=0.6]{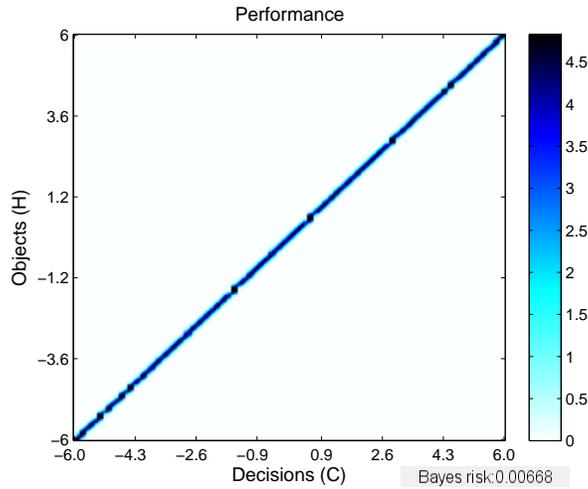}\protect\caption{\label{FigExComputational}Fusion configuration from Proposition \ref{ThmGauss}
with $u=v=1$ and $M=300$.}
\end{figure}

\section{Conclusion\label{SecConclusion}}

We have described a mathematical and computational framework for analyzing
the expected performance of deterministically combining statistical
information under a specified optimality criterion. These results
can be applied to many diverse situations, both in the sensors field
as well as other domains, and can also be extended in a number of
other directions. In particular, many applications involve online
formulations of this problem where the sensor statistics are not known
in advance and need to be estimated from real-time data streams, and
will be explored in future work.

\bibliographystyle{plain}
\bibliography{Fullbib}

\section*{Appendix A: Proofs of Theorems\label{SecAppendix}}
\begin{proof}[Proof of Proposition \ref{ThmDetFusion}]
Let $\mathbf{J}=\prod_{m=1}^{M}J_{m}$ be the joint feature space
with dimension $\mathbf{N}=\sum_{m=1}^{M}N_{m}$. The minimum Bayes
risk is
\begin{eqnarray}
 &  & \inf_{\{C:C(\omega)\in K\}}E(\mathrm{cost}(C,H))\nonumber \\
 & = & \inf_{\mathbf{d}_{C|\mathbf{A}}}\int_{I\times\mathbf{J}\times K}\mathrm{cost}(c,h)\mathbf{d}_{C|\mathbf{A}}(c,\mathbf{a})\left(\prod_{m=1}^{M}\mathbf{d}_{A_{m}|H}(a_{m},h)\right)\mathbf{d}_{H}(h)dc\, d\mathbf{a}\, dh\label{MinCost}
\end{eqnarray}
subject to the constraint that for every $\mathbf{a}\in\mathbf{J}$,
$\mathbf{d}_{C|\mathbf{A}}(\cdot,\mathbf{a})$ is a probability density,
or in other words,
\begin{eqnarray}
\int_{K}\mathbf{d}_{C|\mathbf{A}}(c,\mathbf{a})dc & = & 1,\label{Cond1}\\
\mathbf{d}_{C|\mathbf{A}}(c,\mathbf{a}) & \geq & 0.\label{Cond2}
\end{eqnarray}
In addition, we want the fusion rule to be deterministic, so for all
$\mathbf{a}\in\mathbf{J}$, there is a single finite point $f=f(\mathbf{a})$
in $K$ such that
\begin{equation}
\mathrm{supp}(\mathbf{d}_{C|\mathbf{A}}(\cdot,\mathbf{a}))=\{f(\mathbf{a})\}.\label{Cond3}
\end{equation}
The condition (\ref{Cond3}) makes the problem nonconvex, but it can
be simplified as follows. Let $\mathbf{a}$ be fixed. We want to show
that $\mathbf{d}_{C|\mathbf{A}}(c,\mathbf{a})=\delta(c-f(\mathbf{a}))$.
The condition (\ref{Cond3}) is saying that there is a point $f$
such that $\int_{K}\mathbf{d}_{C|\mathbf{A}}(c,\mathbf{a})\phi(c)dc=0$
for all test functions $\phi\in C_{c}^{0}(K)$ such that $\phi(f)=0$.
Since $\mathbf{d}_{C|\mathbf{A}}$ is compactly supported, this also
holds for the larger class $\phi\in C^{0}(K)$ with $\phi(f)=0$.
Any $\tilde{\phi}\in C^{0}(K)$ can be put into this form by writing
$\tilde{\phi}=\tilde{\phi}-\tilde{\phi}(f)$, and using (\ref{Cond1})
along with linearity implies that $\int_{K}\mathbf{d}_{C|\mathbf{A}}(c,\mathbf{a})\tilde{\phi}(c)dc=\tilde{\phi}(f)$.
Since any generalized function in $D^{0}(K)$ is uniquely determined
by its action on $C_{c}^{0}(K)$ \cite{SW71}, it follows that $\mathbf{d}_{C|\mathbf{A}}(c,\mathbf{a})=\delta(c-f(\mathbf{a}))$,
which implies (\ref{Cond2}) as well. This means that the minimization
problem (\ref{MinCost}) is equivalent to
\begin{equation}
\inf_{\{f:f(\mathbf{J})\subset K\}}\int_{I\times\mathbf{J}}\mathrm{cost}(f(\mathbf{a}),h)\left(\prod_{m=1}^{M}\mathbf{d}_{A_{m}|H}(a_{m},h)\right)\mathbf{d}_{H}(h)\, d\mathbf{a}\, dh,\label{NewMinCost}
\end{equation}
where there are no additional constraints.\\

We first consider the optimization problem (\ref{NewMinCost}) over
the larger space defined by
\begin{equation}
L^{2,H}(\mathbf{J})=\{f:\left\Vert f\right\Vert _{L^{2,H}(\mathbf{J})}=\left(\int_{I\times\mathbf{J}}f(\mathbf{a})^{2}\left(\prod_{m=1}^{M}\mathbf{d}_{A_{m}|H}(a_{m},h)\right)\mathbf{d}_{H}(h)\, d\mathbf{a}\, dh\right)^{1/2}<\infty\}\label{L2Weighted}
\end{equation}
and without any other constraint. In this case, (\ref{NewMinCost})
becomes a standard problem of minimizing a positive quadratic form
over a Hilbert space, with at least one feasible solution since $E(H^{2})<\infty$.
There exists a unique solution to this problem in $L^{2,H}(\mathbf{J})$
\cite{Pa11,SW71}, which is thus unique pointwise up to sets with
zero $\mathbf{N}$-dimensional Lebesgue measure. To find this solution,
we set up the Euler-Lagrange equation \cite{Ev98}
\begin{equation}
\frac{\partial}{\partial f}\int_{I}\frac{1}{2}(f-h)^{2}\left(\prod_{m=1}^{M}\mathbf{d}_{A_{m}|H}(a_{m},h)\right)\mathbf{d}_{H}(h)dh=0,\label{EulerLagrange}
\end{equation}
which can be solved to obtain
\begin{equation}
f(\mathbf{a})=\frac{\int_{I}h\left(\prod_{m=1}^{M}\mathbf{d}_{A_{m}|H}(a_{m},h)\right)\mathbf{d}_{H}(h)dh}{\int_{I}\left(\prod_{m=1}^{M}\mathbf{d}_{A_{m}|H}(a_{m},h)\right)\mathbf{d}_{H}(h)dh}.\label{FusionRuleProof}
\end{equation}
This is the only stationary point of the functional in (\ref{NewMinCost})
and the only candidate for a minimum. The denominator in (\ref{FusionRuleProof})
is simply the joint density $\mathbf{d}_{\mathbf{A}}(\mathbf{a})$,
and since $\mathbf{d}_{A_{m}|H}(a_{m},h)>0$, it follows that $\mathbf{d}_{\mathbf{A}}(\mathbf{a})>0$
as well. If either $\sup(h:h\in K)$ or $\inf(h:h\in K)$ are finite,
we can check that
\begin{eqnarray}
\sup(f(\mathbf{J}):\mathbf{a}\in\mathbf{J}) & \leq & \sup_{\mathbf{a}\in\mathbf{J}}\frac{\int_{I}\sup(h:h\in I)\left(\prod_{m=1}^{M}\mathbf{d}_{A_{m}|H}(a_{m},h)\right)\mathbf{d}_{H}(h)dh}{\int_{I}\left(\prod_{m=1}^{M}\mathbf{d}_{A_{m}|H}(a_{m},h)\right)\mathbf{d}_{H}(h)dh}\nonumber \\
 & = & \sup(h:h\in I)\nonumber \\
 & \leq & \sup(h:h\in K),\label{FusionBound}
\end{eqnarray}
and in the same manner, $\inf f(\mathbf{J})\geq\inf(h:h\in K)$, so
$f$ is a feasible solution for (\ref{NewMinCost}) and is thus in
fact the optimal fusion rule.
\end{proof}
$\quad$
\begin{proof}[Proof of Theorem \ref{ThmPerformance}]
First, note that if for all $c\in K$, the level sets $\{\mathbf{a}\in\mathbf{J}:c=f(\mathbf{a})\}$
have zero $\mathbf{N}$-dimensional measure, then the performance
is given by
\begin{eqnarray}
\mathbf{d}_{C|H}(c,h) & = & \int_{\mathbf{J}}\delta(c-f(\mathbf{a}))\left(\prod_{m=1}^{M}\mathbf{d}_{A_{m}|H}(a_{m},h)\right)d\mathbf{a}.\label{LevelCurve2}
\end{eqnarray}
The integral of (\ref{LevelCurve2}) over $c\in K$ is always $1$
since $\mathbf{d}_{C|H}$ is a probability density. If the level sets
have positive measure, (\ref{LevelCurve2}) is no longer well defined,
but we can still examine a ``smoothed out'' version of (\ref{LevelCurve2})
by considering the Fourier transform (or characteristic function)
of $\mathbf{d}_{C|H}$, given by $E(e^{-2\pi iCz}|H)$. Since $\int_{J_{m}}\mathbf{d}_{A_{m}|H}(a_{m},h)da_{m}=1$,
\begin{eqnarray}
\sup_{z\in\mathbb{R}}\left|\frac{\partial^{l}}{\partial z^{l}}E(e^{-2\pi iCz}|H)\right| & = & \sup_{z\in\mathbb{R}}\left|E((-2\pi iC)^{l}e^{-2\pi iCz}|H)\right|\label{Fourier}\\
 & = & \sup_{z\in\mathbb{R}}\left|\int_{\mathbf{J}}\left(-2\pi if(\mathbf{a})\right)^{l}e^{-2\pi izf(\mathbf{a})}\left(\prod_{m=1}^{M}\mathbf{d}_{A_{m}|H}(a_{m},H)\right)d\mathbf{a}\right|\nonumber \\
 & \leq & \left(2\pi\left\Vert f\right\Vert _{L^{\infty}(\mathbf{J})}\right)^{l}\nonumber \\
 & \leq & \left(2\pi\max(|c|:c\in K)\right)^{l}\nonumber 
\end{eqnarray}
for every $l\geq0$. Now for the $m'$th sensor and $n$th feature
that satisfies (\ref{BalanceCond}), the fact that $\mathbf{d}_{A_{m}|H}$
are all positive on $I$ implies
\begin{eqnarray*}
\left|\frac{\partial f(\mathbf{a})}{\partial(a_{m'})_{n}}\right| & = & \frac{1}{\left|\mathbf{d}_{\mathbf{A}}(\mathbf{a})^{2}\right|}\bigg|\int_{I}\int_{I}\left(\mathbf{d}_{A_{m'}|H}(a_{m'},\tilde{h})\frac{\partial\mathbf{d}_{A_{m'}|H}(a_{m'},h)}{\partial(a_{m'})_{n}}-\mathbf{d}_{A_{m'}|H}(a_{m'},h)\frac{\partial\mathbf{d}_{A_{m'}|H}(a_{m'},\tilde{h})}{\partial(a_{m'})_{n}}\right)\times\\
 &  & \quad\left(\prod_{m\not=m'}\mathbf{d}_{A_{m'}|H}(a_{m},h)\mathbf{d}_{A_{m}|H}(a_{m},\tilde{h})\right)h\mathbf{d}_{H}(h)\mathbf{d}_{H}(\tilde{h})dhd\tilde{h}\bigg|\\
 & > & 0.
\end{eqnarray*}
This means that $\nabla f(\mathbf{a})$ has a nonzero component for
every $\mathbf{a}\in\mathbf{J}$. We form a stationary phase approximation
of the Fourier transform (see \cite{Ev98,Ho90}) by integrating by
parts $l$ times,
\begin{eqnarray}
\left|E(e^{-2\pi iCz}|H)\right| & = & \left|\int_{\mathbf{J}}\frac{1}{\left(-2\pi iz\right)^{l}}\left((T_{f})^{l}e^{-2\pi izf(\mathbf{a})}\right)\left(\prod_{m=1}^{M}\mathbf{d}_{A_{m}|H}(a_{m},H)\right)d\mathbf{a}\right|\nonumber \\
 & = & \frac{1}{\left(2\pi z\right)^{l}}\left|\int_{\mathbf{J}}e^{-2\pi izf(\mathbf{a})}(T_{f}^{*})^{l}\left(\prod_{m=1}^{M}\mathbf{d}_{A_{m}|H}(a_{m},H)\right)d\mathbf{a}\right|\nonumber \\
 & \leq & \frac{1}{\left(2\pi z\right)^{l}}\int_{\mathbf{J}}\left|(T_{f}^{*})^{l}\left(\prod_{m=1}^{M}\mathbf{d}_{A_{m}|H}(a_{m},H)\right)\right|d\mathbf{a}\label{StationaryPhase}
\end{eqnarray}
where $T_{f}$ is the linear differential operator $T_{f}g(\mathbf{a})=\frac{\nabla g(\mathbf{a})\cdot\nabla f(\mathbf{a})}{|\nabla f(\mathbf{a})|^{2}}$
and $T_{f}^{*}g(\mathbf{a})=-\mathrm{div}\left(\frac{g(\mathbf{a})\nabla f(\mathbf{a})}{\left|\nabla f(\mathbf{a})\right|^{2}}\right)$
is its adjoint. The integral in (\ref{StationaryPhase}) is finite
for each $l$ due to the conditions on $\mathbf{J}$ and $\mathbf{d}_{A_{m}|H}$.
This means that for each $h$, $E(e^{-2\pi iCz}|H=h)$ is a smooth
function that decays faster than any polynomial. The inverse Fourier
transform $\mathbf{d}_{C|H}(\cdot,h)$ is thus also a smooth function
with the same decay \cite{SW71}.
\end{proof}
$\quad$
\begin{proof}[Proof of Theorem \ref{ThmDiscFusion}]
To simplify the notation, we let
\[
G(\mathbf{a},h)=\left(\prod_{m=1}^{M}\mathbf{d}_{A_{m}|H}(a_{m},h)\right)\mathbf{d}_{H}(h).
\]
For any fusion rule $g$, the quadratic Bayes risk can be expanded
by writing
\begin{eqnarray}
 &  & \frac{1}{2}\int_{I\times\mathbf{J}}(g(\mathbf{a})-h)^{2}G(\mathbf{a},h)\, d\mathbf{a}\, dh,\nonumber \\
 & = & \frac{1}{2}\int_{I\times\mathbf{J}}\left((g(\mathbf{a})-f^{*}(\mathbf{a}))^{2}+(f^{*}(\mathbf{a})-h)^{2}+2(g(\mathbf{a})-f^{*}(\mathbf{a}))(f^{*}(\mathbf{a})-h)\right)G(\mathbf{a},h)\, d\mathbf{a}\, dh.\label{ThreeTerms}
\end{eqnarray}

For bounded $K$, the third term in (\ref{ThreeTerms}) can be calculated
using the fact that the densities $\mathbf{d}_{A_{m}|H}$ and $\mathbf{d}_{H}$
are all nonnegative,
\begin{eqnarray}
 &  & \left|\int_{I\times\mathbf{J}}2(g(\mathbf{a})-f(\mathbf{a}))(f(\mathbf{a})-h)G(\mathbf{a},h)d\mathbf{a}\, dh\right|\nonumber \\
 & \leq & 2\left\Vert g-f\right\Vert _{L^{\infty}(\mathbf{J})}\int_{\mathbf{J}}\left|\int_{I}(f(\mathbf{a})-h)G(\mathbf{a},h)dh\right|d\mathbf{a}\nonumber \\
 & \leq & 2|K|\int_{\mathbf{J}}\left|\left(\frac{\int_{I}\tilde{h}G(\mathbf{a},\tilde{h})d\tilde{h}}{\int_{I}G(\mathbf{a},\tilde{h})d\tilde{h}}\int_{I}G(\mathbf{a},h)dh-\int_{I}hG(\mathbf{a},h)dh\right)\right|d\mathbf{a}\nonumber \\
 & = & 0.\label{ApproxRisk}
\end{eqnarray}
Since (\ref{ApproxRisk}) is independent of $K$, this in fact holds
for unbounded $K$ as well. The second term in (\ref{ThreeTerms})
is simply the Bayes risk of $f$ itself, so we only need to show that
the first term (the Bayes risk of the additional contribution from
the approximation error) is minimized by the feasible solution $g=f^{*}$.
Since $K$ is closed, $f^{*}(\mathbf{a})\in K$ and minimizes $(f^{*}(\mathbf{a})-f(\mathbf{a}))^{2}$
at each point $\mathbf{a}\in J^{2}$, and the optimality of $f^{*}$
follows from the nonnegativity of $G(\mathbf{a},h)$. Finally, to
show that $f^{*}$ is the unique optimal fusion rule, for any other
optimal fusion rule $g$, (\ref{ThreeTerms}) and (\ref{ApproxRisk})
show that
\[
\frac{1}{2}\int_{I\times\mathbf{J}}(g(\mathbf{a})-f^{*}(\mathbf{a}))^{2}G(\mathbf{a},h)\, d\mathbf{a}\, dh=0,
\]
and since each $\mathbf{d}_{A_{m}|H}(\cdot,h)$ is continuous by assumption,
$g-f^{*}=0$ on $\mathbf{J}$ except possibly on a set of zero $\mathbf{N}$-dimensional
measure.
\end{proof}
$\quad$
\begin{proof}[Proof of Theorem \ref{ThmLpFusion}]
Let $\mathbf{a}\in\mathbf{J}$ be fixed and $f$ be given by (\ref{FusionRule}).
Define $E(h)=h^{2n-1}$ for integer $n>0$ and as before,
\[
G(h)=G(\mathbf{a},h)=\left(\prod_{m=1}^{M}\mathbf{d}_{A_{m}|H}(a_{m},h)\right)\mathbf{d}_{H}(h).
\]
The sub-Gaussian condition on $H$ and the continuity of $\mathbf{d}_{A_{m}|H}$
imply that $\int_{-\infty}^{\infty}e^{rh^{2}}G(h)dh<RE(e^{rH{}^{2}})$
for some constant $R$ depending only on $\mathbf{a}$. The Fourier
transform $\widehat{G}$ satisfies
\begin{eqnarray}
\left|\widehat{G}(z)\right| & \leq & \int_{-\infty}^{\infty}e^{2\pi h|\mathrm{Im}(z)|}e^{-rh^{2}}e^{rh^{2}}G(h)dh\nonumber \\
 & \leq & RE(e^{rH^{2}})\max_{h\in\mathbb{R}}\left(e^{2\pi h|\mathrm{Im}(z)|-rh^{2}}\right)\nonumber \\
 & = & RE(e^{rH^{2}})e^{\frac{\pi^{2}}{r}\mathrm{Im(z)}^{2}}\label{GaussBound}
\end{eqnarray}
for all $z\in\mathbb{C}$. This means that the integral defining $\widehat{G}$
converges uniformly in $z$ on any compact subset of $\mathbb{C}$,
so $\widehat{G}$ is an entire function, and the bound (\ref{GaussBound})
also shows that it has order at most $2$ (see \cite{Al79} for a
definition).

Now since $W$ is even and entire, its Taylor series around the origin
converges everywhere and contains only even powers $z^{2n}$, which
means that the series for $W'$ has the form
\[
W'(z)=\sum_{n=1}^{\infty}w_{n}z^{2n-1}.
\]
The Euler-Lagrange equation for the problem (\ref{NewMinCost}) with
the cost $W(c-h)$ is given by
\begin{eqnarray}
0 & = & \frac{\partial}{\partial f}\int_{I}W(f-h)\left(\prod_{m=1}^{M}\mathbf{d}_{A_{m}|H}(a_{m},h)\right)\mathbf{d}_{H}(h)dh\label{ELGeneral}\\
 & = & \sum_{n=1}^{\infty}w_{n}(E\star G)(f)\nonumber \\
 & = & \sum_{n=1}^{\infty}w_{n}\int_{-\infty}^{\infty}e^{2\pi ifz}\widehat{E}(z)\widehat{G}(z)dz\label{Plancherel}\\
 & = & \sum_{n=1}^{\infty}w_{n}(-2\pi i)^{2n-1}\left(\frac{d}{dz}\right)^{2n-1}\left(e^{2\pi ifz}\widehat{G}(z)\right)\bigg|_{z=0},\label{ELFourier}
\end{eqnarray}
where (\ref{Plancherel}) is justified because the tempered generalized
function $E$ has the (compactly supported) Fourier transform $(-2\pi i)^{2n-1}\delta^{(2n-1)}$,
while $\widehat{G}$ is a smooth function. We want to show that every
term in the sum (\ref{ELFourier}) is zero. Proposition \ref{ThmDetFusion}
says that it is zero if $w_{n}=0$ for $n\geq2$, and solving (\ref{ELFourier})
for $f$ in this case gives
\[
f=-\frac{\widehat{G}'(0)}{2\pi i\widehat{G}(0)},
\]
where $\widehat{G}(0)=\mathbf{d}_{\mathbf{A}}(\mathbf{a})>0$ as in
Proposition \ref{ThmDetFusion}. Now let $\{\lambda_{k}\}_{k\in\mathbb{Z}}$
be the zeros of $\widehat{G}$ in the right half of the complex plane,
indexed in order of increasing absolute value. Since $E(e^{-2\pi izH}|\mathbf{A=a})=\widehat{G}(z)$,
our assumptions say that for each $\lambda_{k}$, $\widehat{G}$ has
another zero $-\lambda_{k}$ in the left half plane. We expand $\widehat{G}$
using the Hadamard factorization theorem \cite{Al79} for functions
of order $2$,
\begin{eqnarray*}
e^{2\pi ifz}\widehat{G}(z) & = & e^{2\pi ifz}\widehat{G}(0)e^{\frac{\widehat{G}'(0)}{\widehat{G}(0)}z+\left(\frac{\widehat{G}''(0)}{\widehat{G}(0)}-\frac{\widehat{G}'(0)^{2}}{\widehat{G}(0)^{2}}\right)\frac{z^{2}}{2}}\prod_{k}\left(1-\frac{z}{\lambda_{k}}\right)e^{\frac{z}{\lambda_{k}}+\frac{1}{2}\left(\frac{z}{\lambda_{k}}\right)^{2}}\left(1+\frac{z}{\lambda_{k}}\right)e^{-\frac{z}{\lambda_{k}}+\frac{1}{2}\left(-\frac{z}{\lambda_{k}}\right)^{2}}\\
 & = & \widehat{G}(0)e^{\left(\frac{\widehat{G}''(0)}{\widehat{G}(0)}-\frac{\widehat{G}'(0)^{2}}{\widehat{G}(0)^{2}}\right)\frac{z^{2}}{2}}\prod_{k}\left(1-\frac{z^{2}}{\lambda_{k}^{2}}\right)e^{\frac{z^{2}}{\lambda_{k}^{2}}}.
\end{eqnarray*}
This formula shows that the mapping $z^{2}\to e^{2\pi ifz}\widehat{G}(z)$
has an analytic branch, or in other words, there is an entire function
$T$ such that $T(z^{2})=e^{2\pi ifz}\widehat{G}(z)$. This means
that the Taylor series expansion of $e^{2\pi ifz}\widehat{G}(z)$
around $z=0$ can only contain even powers $z^{2n}$, $n>0$, with
the odd terms all being zero. Therefore, each term in (\ref{ELFourier})
is zero and $f$ satisfies the Euler-Lagrange equation (\ref{ELGeneral})
for any $W$. We only need to show that $f$ is in fact a minimum
of (\ref{NewMinCost}), which can be done using the standard ``direct
method'' from the calculus of variations (see \cite[p. 443-453]{Ev98}
and \cite{Da89} for details). By the conditions on $W$, the functional
in (\ref{NewMinCost}) is nonnegative and satisfies the bounds
\begin{eqnarray}
 &  & \alpha_{1}\left\Vert f\right\Vert _{L^{p,H}(\mathbf{J})}-\alpha_{1}\int_{I\times\mathbf{J}}\left|h\right|^{p}G(\mathbf{a},h)\, d\mathbf{a}\, dh\nonumber \\
 & \leq & \int_{I\times\mathbf{J}}W(f(\mathbf{a})-h)G(\mathbf{a},h)\, d\mathbf{a}\, dh\label{DMLower}\\
 & \leq & \alpha_{2}\left\Vert f\right\Vert _{L^{p,H}(\mathbf{J})}+\alpha_{2}\int_{I\times\mathbf{J}}\left|h\right|^{p}G(\mathbf{a},h)\, d\mathbf{a}\, dh,\label{DMUpper}
\end{eqnarray}
where $L^{p,H}(\mathbf{J})$ is defined in the same manner as (\ref{L2Weighted}).
The lower bound (\ref{DMLower}) together with the convexity of $W$
imply that the functional (\ref{NewMinCost}) is weakly lower semicontinuous
and has a minimum in $L^{p,H}(\mathbf{J})$ \cite[p. 448]{Ev98}.
On the other hand, the upper bound (\ref{DMUpper}) and the convexity
of $W$ imply that any feasible solution of (\ref{ELGeneral}) is
in fact a minimum of (\ref{NewMinCost}) \cite[p. 452]{Ev98}.
\end{proof}
$\quad$
\begin{proof}[Proof of Theorem \ref{ThmAsymptotic}]
Consider the fusion rule $g$ given by the sample mean $g(\mathbf{a})=\frac{1}{M}\sum_{m=1}^{M}a_{m}$.
Since $\mathrm{var}(A_{m}|H)$ is uniformly bounded, one form of the
law of large numbers \cite{Ci11} shows that $g(\mathbf{A}|H)$ converges
almost surely (and in distribution) to $H$. Since $H\in I$, this
also means that $g(\mathbf{A})\in K$ almost surely for sufficiently
large $M$, so $g$ is in the feasible set of (\ref{NewMinCost}).
For any test function $\phi\in C_{c}^{0}(I\times K)$,
\begin{eqnarray*}
\int_{I}\int_{K}\phi(c,h)\mathbf{d}_{C|H}(c,h)\mathbf{d}_{H}(h)dcdh & = & E(\phi(g(\mathbf{A}),H))\\
 & = & E(E(\phi(g(\mathbf{A}|H),H)))\\
 & \to & E(\phi(H,H))\\
 & = & \int_{I}\phi(h,h)\mathbf{d}_{H}(h)dh,
\end{eqnarray*}
Since $\mathbf{d}_{H}>0$, this is equivalent to saying that
\[
\int_{I}\int_{K}\phi(c,h)\mathbf{d}_{C|H}(c,h)dcdh\to\int_{I}\phi(h,h)dh,
\]
which is the statement that $\mathrm{\mathbf{d}}_{C|H}(c,h)\to\delta(c-h)$
in the weak-$\star$ sense. Under any cost function for which $f$
given by (\ref{FusionRule}) is optimal and $\mathrm{cost}(h,h)=0$,
we end up with
\begin{equation}
E(\mathrm{cost}(f(\mathbf{A}),H))\leq E(\mathrm{cost}(g(\mathbf{A}),H))\to0.\label{BayesZero}
\end{equation}
Now if $\mathbf{A}$ and $H$ satisfy the conditions of Theorem \ref{ThmLpFusion},
then the limit (\ref{BayesZero}) holds for all appropriate costs
$W(c-h)$ with $W(0)=0$. We want to show that the costs $(c-h)^{p}$
for even $p$ are numerous enough to approximate anything in a space
of test functions that $\mathbf{d}_{C|H}$ acts on. Since $\mathbf{d}_{C|H}$
is a bounded and nonnegative function, (\ref{BayesZero}) implies
that for each $h\in I$,
\begin{equation}
\int_{K}W(c-h)\mathbf{d}_{C|H}(c,h)dc\to W(0).\label{BayesW}
\end{equation}
This obviously holds for constant functions $W$ as well, where the
Bayes risk is independent of $f$ or $M$. This means that (\ref{BayesW})
also holds for functions in the linear span
\[
\mathbf{S}=\{W(x)=\sum_{l=0}^{L}d_{l}x^{2l}:x=c-h\in\Lambda,d_{l}\in\mathbb{R}\},
\]
where $\Lambda=\{c-h:c\in K,h\in I,c-h\geq0\}$. Since
\[
\int_{K}W(c-h)\mathbf{d}_{C|H}(c,h)dc\leq\left\Vert W\right\Vert _{L^{\infty}(\Lambda)},
\]
the limit (\ref{BayesW}) holds uniformly over all $W\in\mathbf{S}\cap\mathbf{B}$,
where $\mathbf{B}$ is the closed unit ball $\mathbf{B}=\{W:\left\Vert W\right\Vert _{L^{\infty}(\Lambda)}\leq1\}$.
It is easy to check that $\mathbf{S}$ is a vector space and an associative
algebra (i.e., the product of two functions in $\mathbf{S}$ is also
in $\mathbf{S}$), and for any two distinct points $x_{1}$ and $x_{2}$
in $\Lambda$, we can find $W\in\mathbf{S}$ such that $W(x_{1})\not=W(x_{2})$
by taking $W(x)=x^{2}$. Therefore, the Stone-Weierstrass theorem
\cite{Ru91} shows that $\mathbf{S}$ is dense in $C^{0}(\Lambda)$,
and consequently that $\mathbf{S}\cap\mathbf{B}$ is dense in $C^{0}(\Lambda)\cap\mathbf{B}$.
From (\ref{BayesW}) and the fact that $\mathbf{d}_{C|H}(\cdot,h)$
is even, we conclude that $\mathbf{d}_{C|H}(\cdot,h)\to\delta(\cdot-h)$
for every $h\in I$.
\end{proof}
$\quad$
\begin{proof}[Proof of Proposition \ref{ThmGauss}]
To simplify the notation, we define $A_{m+M/2}=B_{m}$ for $1\leq m\leq M/2$.
This just reflects the fact that two sensors with $M/2$ independent
(given $H$) features each are equivalent to $M$ sensors with one
such feature each, or to a single sensor with $M$ such features.
We have
\begin{eqnarray*}
f(\mathbf{A}) & = & \frac{\int_{-\infty}^{\infty}h\prod_{m=1}^{M}\left((2\pi v)^{-1/2}e^{-(A_{m}-uh)^{2}/(2v)}\right)\left((2\pi)^{-1/2}e^{-h^{2}/2}\right)dh}{\int_{-\infty}^{\infty}\prod_{m=1}^{M}\left((2\pi v)^{-1/2}e^{-(A_{m}-uh)^{2}/(2v)}\right)\left((2\pi)^{-1/2}e^{-h^{2}/2}\right)dh}\\
 & = & \frac{\int_{-\infty}^{\infty}(2\pi)^{-1/2}he^{-\left(\frac{Mu^{2}}{2v}+\frac{1}{2}\right)h^{2}+\frac{u}{v}\sum_{m=1}^{M}A_{m}h}dh}{\int_{-\infty}^{\infty}(2\pi)^{-1/2}e^{-\left(\frac{Mu^{2}}{2v}+\frac{1}{2}\right)h^{2}+\frac{u}{v}\sum_{m=1}^{M}A_{m}h}dh}\\
 & = & \frac{\left(\sum_{m=1}^{M}A_{m}\right)\exp\left(\frac{\left(\sum_{m=1}^{M}A_{m}u\right)^{2}}{2(Mu^{2}v+2v^{2})}\right)uv^{1/2}(Mu^{2}+v)^{-3/2}}{\exp\left(\frac{\left(\sum_{m=1}^{M}A_{m}u\right)^{2}}{2(Mu^{2}v+2v^{2})}\right)v^{1/2}(Mu^{2}+v)^{-1/2}}\\
 & = & \frac{u\sum_{m=1}^{M}A_{m}}{Mu^{2}+v}.
\end{eqnarray*}
This is (\ref{GaussFusion}). We next use the Fourier transform approach
from Theorem \ref{ThmPerformance} to find the performance of this
fusion rule. First,
\begin{eqnarray*}
E(e^{-2\pi izC}|H) & = & \int_{\mathbb{R}^{M}}e^{-2\pi iz\left(\frac{u}{Mu^{2}+v}\sum_{m=1}^{M}a_{m}\right)}\prod_{m=1}^{M}\left((2\pi v)^{-1/2}e^{-(a_{m}-uH)^{2}/(2v)}da_{m}\right).\\
 & = & \left(\int_{-\infty}^{\infty}e^{-2\pi iz\frac{ua_{m}}{Mu^{2}+v}}(2\pi v)^{-1/2}e^{-(a_{m}-uH)^{2}/(2v)}da_{m}\right)^{M}\\
 & = & \exp\left(-\frac{M}{(Mu^{2}+v)^{2}}\left(2\pi^{2}u^{2}vz^{2}+2\pi iu^{2}(Mu^{2}+v)Hz\right)\right).
\end{eqnarray*}
Some straightforward (although technical) calculations give
\begin{eqnarray*}
\mathbf{d}_{C|H}(c,h) & = & \int_{-\infty}^{\infty}E(e^{-2\pi izC}|H=h)e^{2\pi izc}dz\\
 & = & \frac{Mu^{2}+v}{u\sqrt{2\pi Mv}}\exp\left(-\frac{\left(Mu^{2}(c-h)+cv\right)^{2}}{2Mu^{2}v}\right)
\end{eqnarray*}
and
\begin{eqnarray*}
E((C-H)^{2}) & = & \int_{-\infty}^{\infty}\int_{-\infty}^{\infty}(c-h)^{2}\mathbf{d}_{C|H}(c,h)(2\pi)^{-1/2}e^{-h^{2}/2}dcdh\\
 & = & \frac{v}{u\sqrt{2\pi M(Mu^{2}+v)}}\int_{-\infty}^{\infty}e^{-\frac{Mu^{2}+v}{2Mu^{2}}c^{2}}dc\\
 & = & \frac{v}{Mu^{2}+v}.
\end{eqnarray*}
Finally, a similar calculation gives
\[
E(e^{-2\pi izH}|\mathbf{A})=\sqrt{\frac{v}{(Mu^{2}+v)(2\pi v)^{M}}}\exp\left(\frac{\left(2\pi ivz-u\sum_{m=1}^{M}A_{m}\right)^{2}}{2v(Mu^{2}+v)}-\frac{\left(\sum_{m=1}^{M}A_{m}\right)^{2}}{2v}\right),
\]
so $E(e^{-2\pi izH}|\mathbf{A})$ has no zeros in the complex plane
for any $\mathbf{A}$, and the other conditions of Theorem \ref{ThmLpFusion}
obviously hold.
\end{proof}
$\quad$
\begin{proof}[Proof of Corollary \ref{CorGaussSep}]
By Proposition \ref{ThmGauss} (\ref{GaussFusion}), the locally
optimal fusion rules at $A^{\star}$ and $B^{\star}$ are given by
$A^{\star}=\frac{\sum_{m=1}^{M/2}A_{m}}{M/2+1}$ and $B^{\star}=\frac{\sum_{m=1}^{M/2}B_{m}}{M/2+1}$.
These are Gaussian variables with means $\frac{M}{M+2}h$ and variances
$\frac{2M}{(M+2)^{2}}$. Using these inputs, we apply (\ref{GaussFusion})
again to determine the locally optimal system fusion rule $C$. The
fusion rule is $C=\frac{M+2}{2M+2}(A^{\star}+B^{\star})$ and its
Bayes risk is $\frac{M+2}{2M+2}\frac{\frac{2M}{(M+2)^{2}}}{\frac{M}{M+2}}=\frac{1}{M+1}$.
\end{proof}
$\quad$
\begin{proof}[Proof of Proposition \ref{ThmExpo}]
From the standard integral definition of the gamma function, the
fusion rule is
\begin{eqnarray*}
f(\mathbf{A}) & = & \frac{\int_{0}^{\infty}h\prod_{m=1}^{M}\left(he^{-A_{m}h}\right)e^{-h}dh}{\int_{0}^{\infty}\prod_{m=1}^{M}\left(he^{-A_{m}h}\right)e^{-h}dh}=\frac{\left(\sum_{m=1}^{M}A_{m}+1\right)^{-(M+2)}(M+1)!}{\left(\sum_{m=1}^{M}A_{m}+1\right)^{-(M+1)}M!}=\frac{M+1}{\sum_{m=1}^{M}A_{m}+1}.
\end{eqnarray*}
Note that we always have $C=f(\mathbf{A})\in(0,M+1]$. To determine
its performance, the formula (\ref{LevelCurve2}) is easier to apply
directly instead of taking the Fourier transform.
\begin{eqnarray*}
\mathbf{d}_{C|H}(c,h) & = & \int_{\mathbf{J}}\delta\left(c-f(\mathbf{a})\right)h^{M}e^{-h\sum_{m=1}^{M}a_{m}}d\mathbf{a}\\
 & = & \int_{\sum_{m=1}^{M}a_{m}=\frac{M+1}{c}-1}\frac{d}{dc}h^{M}e^{-h(\frac{M+1}{c}-1)}|d\mathbf{a}|\\
 & = & (M+1)h^{M+1}c^{-2}e^{-h(\frac{M+1}{c}-1)}\int_{0}^{\frac{M+1}{c}-1}\int_{0}^{a_{M}}...\int_{0}^{a_{3}}\int_{0}^{a_{2}}da_{1}da_{2}...da_{M-1}da_{M}\\
 & = & \frac{(M+1)h^{M+1}}{M!c^{2}}\left(\frac{M+1}{c}-1\right)^{M}e^{-h(\frac{M+1}{c}-1)}
\end{eqnarray*}

The quadratic Bayes risk can be evaluated in a similar way,
\begin{eqnarray*}
 &  & E((C-H)^{2})\\
 & = & \int_{0}^{M+1}\frac{M+1}{M!}\left(\frac{M+1}{c}-1\right)^{M}\int_{0}^{\infty}(c^{-2}h^{M+1}-2c^{-1}h^{M+2}+h^{M+3})e^{-h\frac{M+1}{c}}dhdc\\
 & = & \int_{0}^{M+1}\frac{M+1}{M!}\left(\frac{M+1}{c}-1\right)^{M}\left(\frac{c^{M+2}(M+1)!}{(M+1)^{M+2}}-\frac{2c^{M+2}(M+2)!}{(M+1)^{M+3}}+\frac{c^{M+4}(M+3)!}{(M+1)^{M+4}}\right)dc\\
 & = & \frac{M+3}{(M+1)^{M+2}}\int_{0}^{M+1}c^{2}(M+1-c)^{M}dc\\
 & = & \frac{2}{M+2}.
\end{eqnarray*}
\end{proof}

\end{document}